\documentclass[11pt, a4paper]{article}

\usepackage{lipsum}
\usepackage{geometry}
\usepackage{authblk}
\usepackage{bm}
\usepackage{amssymb, amsmath}
\usepackage{color}
\usepackage[amsmath,thmmarks]{ntheorem}
\numberwithin{equation}{section}
\newcommand{\fnum}{\mathbb{F}}
\DeclareMathOperator{\tr}{tr}
\DeclareMathOperator{\norm}{N}
\DeclareMathOperator{\moda}{mod}
\newcommand{\bases}[1]{\langle #1 \rangle}
\DeclareMathOperator{\re}{Re}
\newtheorem{theorem}{Theorem}[section]

\newtheorem{lemma}{Lemma}[section]
\newtheorem{corollary}{Corollary}[section]

\theorembodyfont{\rmfamily}

\newtheorem{definition}{Definition}[section]

\theoremstyle{nonumberplain}
\theoremsymbol{$\square$}
\newtheorem{proof}{Proof}

\begin{document}

\title{On the Polynomial Ramanujan Sums \\over Finite Fields}

\date{}

\author{Zhiyong Zheng\footnote{This work was partially supported by the ``973'' project 2013CB834205.}}
\affil{School of Mathematics and Systems Science, \authorcr
Beihang University, Beijing, P.R. China \authorcr
zhengzhiyong@buaa.edu.cn \vspace{4mm}}

\maketitle
\begin{abstract}
The polynomial Ramanujan sum was first introduced by Carlitz \cite{ref:7}, and a generalized version by Cohen \cite{ref:10}. In this paper, we study the arithmetical and analytic properties of these sums, derive various fundamental identities, such as H\"{o}lder formula, reciprocity formula, orthogonality relation and Davenport--Hasse type formula. In particular, we show that the special Dirichlet series involving the polynomial Ramanujan sums are, indeed, the entire functions on the whole complex plane, we also give a square mean values estimation. The main results of this paper are new appearance to us, which indicate the particularity of the polynomial Ramanujan sums. 
\end{abstract}



\section{Introduction and Results Statement}
One hundred years ago, Ramanujan was the first to appreciate the importance of the following exponential sum
\begin{equation}
  \label{eq:expSum}
  c(m, n) = \sum_{\substack{d=1, \\(d,n)=1}}^{n}e(\frac{md}{n}),
\end{equation}
where $e(x) = e^{2\pi i x}$, $n$ and $m$ are positive integers. His interest in this sum originated in his desire to obtain expression for a variety of well-known arithmetical functions in the form of a series $\sum_n a_nc(m,n)$, in particular, he obtained some very famous identities, for example (see \cite{ref:39})
\begin{equation}
  \label{eq:famousIden}
  \sum_{n=1}^{+\infty}\frac{c(m,n)}{n}=0, \quad \sum_{n=1}^{+\infty}\frac{c(m,n)}{m}=-\Lambda(n), \mbox{ and } c(m,n)=\sum_{d|n, d|m}d \mu(\frac{n}{d}),
\end{equation}
where $\mu(n)$ is the M\"{o}bius function, $\Lambda(n)$ is the Mangoldt function, and the last equality is usually said to be the Kluyer's equality. Following him, many other authors were also interested in this fascinating sum \cite{ref:6,ref:7,ref:34}, especially it makes surprising appearances in singular series of the Hardy--Littlewood asymptotic formula for Waring problems and in the asymptotic formula of Vingradov on sums of three primes. For details the reader is referred to \cite{ref:18}. 

Many mathematicians later tried to generalize this sum to find more and more applications. One of the most popular generalization was given by E. Cohen in \cite{ref:10,ref:11,ref:12} that
\begin{equation}
  \label{eq:AppCohen}
  c_k(m, n) = \sum_{\substack{d=1, \\(d,n^k)_k=1}}^{n}e(\frac{md}{n^k}),
\end{equation}
where the g.c.d function $(a,b)_k$ is the greatest common $k$-th power divisors of the integers $a$ and $b$. In \cite{ref:10} and \cite{ref:12}, Cohen presented an analogue of the Kluyer and H\"{o}lder formula for the above generalized Ramanujan sums
\begin{equation}
  \label{eq:AppCohen2}
  c_k(m, n) = \sum_{d|n, d^k|m}d^k \mu(\frac{n}{d}) = \phi_k(n)\mu(N)\phi_k^{-1}(N),
\end{equation}
where $N^k = \frac{n^k}{(m, n^k)_k}$, and $\phi_k(n)$ is the Jordan totient function given by the following product expression ($z$ is a complex number)
\begin{equation}
  \label{eq:AppCohen3}
  \phi_z(n) = n^z \prod_{\substack{p|n\\ p ~prime}}(1-\frac{1}{p^z}).
\end{equation}

Various other generalization were discussed in many papers including that of \cite{ref:3,ref:9,ref:30,ref:33,ref:40,ref:48}. Here we mention another interesting result, namely reciprocity formula. In \cite{ref:24,ref:25}, K. Johenson showed that
\begin{equation}
  \label{eq:recip}
  \frac{\mu(\bar{m})c(nm^*, m)}{m^*} = \frac{\mu(\bar{n})c(mn^*, n)}{n^*},
\end{equation}
where $\bar{n}$ denotes the largest square-free divisions of $n$, and $n^* = \frac{n}{\bar{n}}$. 

In more recent years, people have more and more interests in this sum, it appeared in various other seemingly unrelated problems. In Algebra, Ramanujan sums as the super characters exhibit a new application of the theory of super characters \cite{ref:20}, which recently developed by Diaconis--Issacs and Andre' \cite{ref:15}, and Ramanujan sums in arithmetical semigroup \cite{ref:22}. In number theory, Ramanujan sums appeared in the study of Waring type formula \cite{ref:28}, the distribution of rational numbers in short interval \cite{ref:26}, equire partition modulo odd integers \cite{ref:5}, the certain weighted average \cite{ref:1,ref:2,ref:8,ref:21,ref:32,ref:42,ref:47}, graph theory \cite{ref:17}, symmetry classes of tensors \cite{ref:45}, combinatorics \cite{ref:41}, cyclotomic polynomials \cite{ref:19,ref:31,ref:46}, and Mahler matrices \cite{ref:29}. In physics, Ramanujan sums have applications in the processing of low-frequency noise \cite{ref:37}, and of long-period sequences \cite{ref:36}, and in the study of quantum phase locking \cite{ref:38}. 

The main purpose of this paper is to generalize the Ramanujan sums to the polynomial case, and discuss various analogue properties of the classical case. The polynomial Ramanujan sum was first introduced by Carlitz in \cite{ref:7}, and a generalized version by E. Cohen in \cite{ref:10}. To state their definition, let $\fnum_q$ be a finite field with $q=p^l$ elements, where $p$ is a prime number, $\fnum_q[x]$ be the polynomial ring. Suppose $H$ is a fixed polynomial in $\fnum_q[x]$, and $H \neq 0$, we first choose a complex-valued character of the additive group of the residue class ring $\fnum_q[x]/\bases{H}$, these characters are said to be the additive characters modulo $H$ on $\fnum_q[x]$. If $A$ is in $\fnum_q[x]$, let $A \equiv a_{m-1}x^{m-1} + \cdots + a_1x + a_0 ~(\moda H)$, where $m=\deg(H)$, we set an additive function modulo $H$ on $\fnum_q[x]$ by $t(A) = a_{m-1}$. Then, for any $A$ and $B$ in $\fnum_q[x]$, we have $t(A+B) = t(A) + t(B)$, and $t(A) = t(B)$ if $A \equiv B~(\moda H)$, in particular, $t(A) = 0$ whenever $H|A$. To generalize this $t$-function modulo $H$, for any given polynomial $G$ in $\fnum_q[x]$, we let $t_G(A) = t(GA)$. Clearly, $t_G$ is also an additive function modulo $H$, that is
\begin{equation}
  \label{eq:tG}
  t_G(A+B) = t_G(A) + t_G(B), \mbox{ and } t_G(A) = t_G(B) \mbox{ if } A \equiv B~(\moda H).
\end{equation}

Next, let $\lambda$ be a fixed non-principal character on $\fnum_q$, for example, one may choose $\lambda(a) = e(\frac{\tr(a)}{p})$ for $a\in \fnum_q$, where $\tr(a)$ is the trace map from $\fnum_q$ to $\fnum_p$. We define a complex-valued function $E(G, H)$ on $F_q[x]$ by 

\begin{equation}
  \label{eq:EGH}
  E(G, H)(A) = \lambda(t_G(A)).
\end{equation}
It is easy to see that $E(G, H)$ is an additive character modulo $H$ on $\fnum_q[x]$, and 
\begin{equation}
  \label{eq:EGHprop}
  E(G, H)(A) = E(A, H)(G). 
\end{equation}
The polynomial Ramanujan sum modulo $H$ on $\fnum_q[x]$ is given by (see Carlitz \cite[4.1]{ref:7})
\begin{equation}
  \label{eq:rama}
  \eta(G, H) =   \sum_{\substack{D \moda H, \\(D,H)=1}} E(G, H)(D),
\end{equation}
where the summation extends over a complete residue system modulo $H$ in $\fnum_q[x]$. If $k\geq 1$ is a fixed integer, the generalized version of Cohen (see \cite[3.3]{ref:10}) is the following
\begin{equation}
  \label{eq:ramaGen}
  \eta_k(G, H) =   \sum_{\substack{D \moda H^k, \\(D,H^k)_k=1}} E(G, H^k)(D),
\end{equation}
where the summation ranges over a complete residue system modulo $H^k$ in $\fnum_q[x]$, and the g.c.d. function $(A, B)_k$ denotes the largest $k$-th power common divisor (monic) of the polynomials $A$ and $B$ in $\fnum_q[x]$. We set $(A, B)_1 = (A, B)$ the usual g.c.d.\ function. If $D \moda H^k$ with $(D,H^k)_k=1$, which is said to be a $k$-reduced residue system modulo $H$ according to Cohen \cite{ref:11,ref:12}. Clearly, $\eta_1(G, H) = \eta(G, H)$. 

By the above notations, it is easy to verify (see \cite[(3.4) and (3.5)]{ref:10}) that 
\begin{equation}
  \label{eq:ramaPro1}
  \eta_k(G_1, H) = \eta_k(G_2, H), \mbox{ if } G_1 \equiv G_2~(\moda H^k),
\end{equation}
\begin{equation}
  \label{eq:ramaPro2}
  \eta_k(G, H_1H_2) = \eta_k(G, H_1)\eta_k(G, H_2), \mbox{ if } (H_1, H_2)=1,
\end{equation}
and 
\begin{equation}
  \label{eq:ramaPro3}
  \eta_k(G, H) = \sum_{D|H, D^k|G}|D|^k \mu(\frac{H}{D}), 
\end{equation}
where and later, $\mu(H)$ is the M\"{o}bius function on $\fnum_q[x]$, $|D| = q^{\deg(D)}$ is the absolute value function on $\fnum_q[x]$, and $D|H$ means $D$ is a monic divisor of $H$, and $\sum_{D|H}$ means $D$ extending over all of monic divisors of $H$. 

In other hands, we note that the additive characters $E(G,H)$ given by \eqref{eq:EGH} are, indeed, all of the additive characters $\psi$ modulo $H$, in other words, for any additive character $\psi$ modulo $H$, there exists a unique polynomial $G$ in $\fnum_q[x]$, such that $\psi = E(G, H)$, and $\deg(G) < \deg(H)$ (See Lemma \ref{lem:21} below). Therefore, the polynomial Ramanujan sums $\eta(G, H)$ and $\eta_k(G,H)$ coincide with the classical sums $c(m,n)$ and $c_k(m,n)$ respectively. 

The first result of this paper is to derive an analogue of H\"{o}lder formula for the polynomial sums. There is no essential difficulty to do this, but we make use of a simpler method to show (see Theorem~\ref{thm:holder} below) that
\begin{equation}
  \label{eq:Holder}
  \eta(G, H) = \phi(H)\mu\left(\frac{H}{(G, H)}\right)\phi^{-1}\left(\frac{H}{(G, H)}\right), \mbox{ and } \eta_k(G, H)=\phi_k(H)\mu(N)\phi_k^{-1}(N),
\end{equation}
where $\phi(H)$ is the Euler totient function, and $\phi_k(H)$ is the Jordan totient function on $\fnum_q[x]$, and $N^k = \frac{H^k}{(G, H)_k}$. 

The second result is to present an analogue of the reciprocity formula for the polynomial Ramanujan sums (see Theorem~\ref{thm:reci1} below). Let $\bar{H}$ be the largest square-free divisor of $H$, and $H^* = \frac{H}{\bar{H}}$, then we have 
\begin{equation}
  \label{eq:reci}
  \frac{\mu(\bar{H})\eta(GH^*, H)}{|H^*|} = \frac{\mu(\bar{G})\eta(HG^*, G)}{|G^*|}. 
\end{equation}

The generalized sums $\eta_k(G, H)$ seemingly cannot share this kind of formula when $k>1$, however we derive the second reciprocity formula for $\eta_k(G, H)$ (see Theorem~\ref{thm:reciK} below): if $D_1^{k}|H$ and $D_2^k | H$, then
\begin{equation}
  \label{eq:reciK}
  \phi_k(D_2)\eta_k(\frac{H}{D_2^k}, D_1) =  \phi_k(D_1)\eta_k(\frac{H}{D_1^k}, D_2).
\end{equation}

The importance of (1.17) lies in the fact that it is equivalent to the H\"{o}lder formula (1.15), and plays a role in the proof of the orthogonal relation formula (4.13) below. 

The main results of this paper are the following theorems on the special Dirichlet series involving in the polynomial Ramanujan sums. Let $\mathbb{A}$ be the set of monic polynomials of $\fnum_q[x]$. We define 
\begin{equation}
  \label{eq:Diri}
  \delta_k(s, G) = \sum_{H \in \mathbb{A}}\frac{\eta_k(G, H)}{|H|^s} = \sum_{n=0}^{+\infty}A(n)q^{-ns}, 
\end{equation}
and
\begin{equation}
\label{eq:Diri2}
\tag{1.18'}
  \tau_k(s, G) = \sum_{G \in \mathbb{A}}\frac{\eta_k(G, H)}{|G|^s} = \sum_{n=0}^{+\infty}B(n)q^{-ns}, 
\end{equation}
where $s = c+it$ is a complex number, and 
\begin{equation*}
  A(n) = \sum_{\substack{H\in \mathbb{A} \\ \deg(H)=n}}\eta_k(G, H), \quad B(n) = \sum_{\substack{G\in \mathbb{A} \\ \deg(G)=n}}\eta_k(G, H)
\end{equation*}

The last two infinite series in \eqref{eq:Diri} and \eqref{eq:Diri2} are the definitions of $\delta_k(s, G)$ and $\tau_k(s, H)$, which tell us how to understand the special Dirichlet series in the middle of \eqref{eq:Diri} and \eqref{eq:Diri2}. We show that
\begin{theorem}\label{thm:11}
  If $H^k\nmid G$, then $\delta_k(s, G)$ is an entire function on the whole complex plane, and we have
  \begin{equation}
    \label{eq:thm11-1}
    \delta_k(s, G) = (1-q^{1-s})\sum_{D^k|G}|D|^{k-s}. 
  \end{equation}
\end{theorem}
In particular, we have
\begin{equation}
  \label{eq:thm11-2}
    \delta_k(1, G) = \sum_{H \in \mathbb{A}}\frac{\eta_k(G,H)}{|H|}=0,
\end{equation}
and for any real numbers $c$ and $T>0$, we also have the following square mean value estimation
\begin{equation}
  \label{eq:thm11-3}
  \frac{1}{2T}\int_{-T}^{T}|\delta_k(c+it, G)|^2 dt = (1+q^{2(1-c)})\sigma_0(2(k-c), G) - 2q^{1-k}\sigma_1(2(k-c), G)+O(\frac{1}{T}),
\end{equation}
where ($x$ is a real number)
\begin{equation}
  \label{eq:thm11-4}
  \sigma_0(x, G) = \sum_{D^k|G}|D|^x, \mbox{ and }\sigma_1(x, G)=\sum_{\substack{D_1^k|G, D_2^k|G\\ \deg(D_1) = \deg(D_2)+1}}|D_1|^x,
\end{equation}
and the constant implied by ``$O$" depends on $q$, $G$, and $c$ only. 

If $H^k|G$, then $E(G, H^k)$ is the principal additive character modulo $H^k$, and $\eta_k(G, H) = \phi_k(H)$ by (1.11). It is easy to see that
\begin{equation}
  \label{eq:thm11-5}
  \delta_k(s, G) = \zeta_{\mathbb{A}}^{-1}(s) \zeta_{\mathbb{A}}(s-k), \mbox{ if } \re(s)>k+1,
\end{equation}
where $\zeta_{\mathbb{A}}$ is the zeta function on $\fnum_q[x]$ given by 
\begin{equation}
  \label{eq:zeta}
  \zeta_{\mathbb{A}}(s) = \sum_{H\in \mathbb{A}}\frac{1}{|H|^s}.
\end{equation}
It is well-known (see \cite[Chapter 2]{ref:43}) that
\begin{equation}
  \label{eq:zeta2}
  \zeta_{\mathbb{A}}(s) = (1-q^{1-s})^{-1}, \mbox{ if } \re(s)>1.
\end{equation}
Therefore, $\eta_k(s, G)$ has a simple pole $s=k+1$ with residue $\frac{1}{\zeta_{\mathbb{A}}(k+1)\log q}$ at $s=k+1$, when $H^k|G$.

\begin{theorem}\label{thm:12}
  If $H$ is a positive degree polynomial in $\fnum_q[x]$, then $\tau_k(s, H)$ is an entire function on the whole complex plane, and we have
  \begin{equation}
    \label{eq:thm12-1}
    \tau_k(s, H) = (1-q^{1-s})^{-1}\phi_{k(1-s)}(H),
  \end{equation}
where $\phi_{k(1-s)}(H)$ is the generalized Jordan totient function given by 
\begin{equation}
  \label{eq:thm12-2}
  \phi_z(H) = |H|^z \prod_{\substack{P|H\\ P ~irreducible}}(1-\frac{1}{|P|^z}).
\end{equation}
In particular, we have
\begin{equation}
  \label{eq:thm12-3}
  \tau_k(1, H) = \sum_{G\in \mathbb{A}}\frac{\eta_k(G, H)}{|G|} = -\frac{k\Lambda(H)}{\log q},
\end{equation}
where $\Lambda(H)$ is the Mangoldt function on $\fnum_q[x]$. Moreover, for any real numbers $c$ and $T$ with $c\neq 1$ and $T>0$ we have
\begin{equation}
  \label{eq:thm12-4}
  \begin{split}
& \frac{1}{2T}\int_{-T}^{T}|\tau_k(c+it, H)|^2 dt = |1-q^{2c(1-c)}|^{-1}\sum_{D|H}\mu^2(\frac{H}{D})|D|^{2(1-c)} + 
\\&2|1-q^{2(1-c)}|^{-1}\sum_{\substack{D_1|H, D_2|H\\\deg(D_1)>\deg(D_2)}} \mu(\frac{H}{D_1})\mu(\frac{H}{D_2})|D_1|^{(k+2)(1-c)}|D_2|^{(2-k)(1-c)} + O(\frac{1}{T}),    
  \end{split}   
\end{equation}
where the constant implied by ``$O$" depends on $q$, $H$, and $c$ only. 
\end{theorem}

If $\deg(H)=0$, then $\eta_k(G, H) = 1$ for any $G$ in $\fnum_q[x]$, it follows that $\tau_k(s, H) = \zeta_{\mathbb{A}}(s)$, which has a simple pole $s=1$ with residue $\frac{1}{\log q}$ at $s=1$. If $c=1$, the square mean value estimation is more complicated, and we will present it in another place. 

The polynomial Ramanujan sum $\eta(G, H)$ in fact is a special Gauss sums on $\fnum_q[x]$. In \cite{ref:49}, we presented an analogue of Davenport--Hasse's theorem for the polynomial Gauss sums (see \cite[Theorem 1.3]{ref:49}). In the last section of this paper, we show that the generalized polynomial Ramanujan sums $\eta_k(G, H)$ also share this kind of Davenport--Hasse's formula (see Theorem 7.1 below). 

Throughout this paper, $P$ denotes an irreducible polynomial in $\fnum_q[x]$, $D|H$ means that $D$ is a monic divisor of $H$, $\sum_{D|H}$ means $D$ extending over all of monic divisors of $H$, and $|H| = q^{\deg(H)}$ is the absolute value function on $\fnum_q[x]$. 

\section{Preliminaries}
\label{sec:pre}

We start this section by determining the construction of the additive character group modulo $H$ on $\fnum_q[x]$ via $E(G, H)$. 

\begin{lemma}\label{lem:21}
  For any $\psi$, an additive character modulo $H$ on $\fnum_q[x]$, there exists a unique polynomial $G$ in $\fnum_q[x]$, such that $\psi=E(G, H)$, and $\deg(G)< \deg(H)$. 
\end{lemma}

\begin{proof}
  For the convenient sake, we write $\psi_G = E(G, H)$. By \eqref{eq:EGH}, we have $\psi_{G_1} = \psi_{G_2}$, if $G_1 \equiv G_2~(\moda H)$, hence we may set $G$ in a complete residue system modulo $H$ in $\fnum_q[x]$, so that $\deg(G)<\deg(H)$. Moreover, we have
  \begin{equation}
    \label{eq:lemma21p}
    \psi_{G_1 + G_2} = \psi_{G_1} \cdot \psi_{G_2}, \mbox{ and } \bar{\psi}_G = \psi_{-G},
  \end{equation}
where $\bar{\psi}_G$ is the usual conjugation of a complex number $\psi_G$. Since $\psi_G = \psi_0$, the principal additive character modulo $H$, if $G=0$, or $H|G$. Conversely, we have $\psi_G = \psi_0$ if and only if $H|G$. To show that statement, we see that $\lambda$ is a non-principal additive character on $\fnum_q$ by assumption, then there is an element $a$ in $\fnum_q$, so that $\lambda(a)\neq 1$. If $H\nmid G$, we let
\begin{equation}
  \label{eq:lemma21p2}
  R = (G, H) = x^k + a_{k-1}x^k + \cdots + a_1x + a_0\in \fnum_q[x], 
\end{equation}
where $0\leq k \leq m-1$, and $m=\deg(H)$. It follows that
\begin{equation*}
  ax^{m-1-k}R = ax^{m-1}+ \cdots.
\end{equation*}
We note the following congruent equation in variable $T$ that
\begin{equation}
  \label{eq:lemma21p3}
  GT \equiv ax^{m-1-k}R ~(\moda H)
\end{equation}
is solvable in $\fnum_q[x]$, therefore there exists a polynomial $A$ in $\fnum_q[x]$ such that
\begin{equation}
  \label{eq:lemma21p4}
  GA \equiv ax^{m-1}+\cdots. ~(\moda H),
\end{equation}
and $t_G(A) = t(GA) = a$, so $\psi_G(A) = \lambda(t_G(A)) = \lambda(a)\neq 1$, and $\psi_G \neq \psi_0$. By \eqref{eq:lemma21p}, we have immediately
\begin{equation}
  \label{eq:lemma21p5}
  \psi_{G_1}\neq \psi_{G_2}, \mbox{ if } G_1 \not \equiv G_2~(\moda H)
\end{equation}
because if $\psi_{G_1} = \psi_{G_2}$ then $\psi_{G_1-G_2} = \psi_0$, and $G_1 \equiv G_2~(\moda H)$. This shows that $\psi_G$ are different from each other when $G$ running through a complete residue system of modulo $H$. Hence there are exactly $|H|=q^m$ different characters $\psi_G$, but the number of additive characters modulo $H$ on $\fnum_q[x]$ is exactly $q^m$, thus every character $\psi$ is just of the form $\psi_G$. We complete the proof of Lemma \ref{lem:21}. 
\end{proof}

Next two lemmas are not new, one may find them in Carlitz \cite{ref:7} (see \cite{ref:7}, (2.4), (2.5), and (2.6)), but we give a more explicit expression here. 
\begin{lemma}\label{lem:carlitz}
  If $A$ is a monic polynomial in $\fnum_q[x]$, then we have
  \begin{equation}
    \label{eq:lem22}
    E(GA, HA) = E(G, H).
  \end{equation}
\end{lemma}

\begin{proof}
  For any $B \in \fnum_q[x]$, let 
$$GB \equiv a_{m-1}x^{m-1} + \cdots + a_1x+a_0~(\moda H),$$
then 
$$AGB \equiv A(a_{m-1}x^{m-1} + \cdots +a_0)~(\moda AH).$$
Because $A$ is a monic polynomial, we see that the function $t_{GA}$ modulo $HA$ just is the function $t_G$ modulo $H$. It follows that
\begin{equation}
  \label{eq:lem22main}
  E(GA, HA)(B) = \lambda(t_G(B)) = E(G, H)(B),
\end{equation}
and the lemma follows immediately.
\end{proof}

\begin{lemma}\label{lem:twoCases}
  Suppose $A\in \fnum_q[x]$, then we have
  \begin{equation}
    \sum_{G\moda H}E(G, H)(A) = \left\{
      \begin{tabular}[l]{l}
        $|H|$, \mbox{ if }$H|A$,\\
        $0$, \mbox{otherwise},
      \end{tabular}\right.
  \end{equation}
where the summation extends over a complete residue system modulo $H$. 
\end{lemma}

\begin{proof}
  By Lemma \ref{lem:21}, it just is the orthogonal relation formula, we have the lemma at once.  
\end{proof}

The following lemma is a more complicated orthogonal relation, which will be used in the next section. 

\begin{lemma}\label{lem:24}
  If $D_1 | H$, $D_2 |H$, $(X, D_1^k)_k = 1$, and $(Y, D_2^k)_k=1$, where $X$ and $Y$ are two polynomials in $\fnum_q[x]$ with $\deg(X)<k\deg(D_1)$ and $\deg(Y)<k\deg(D_2)$, then we have
  \begin{equation}
    \label{eq:lemma24}
    \begin{split}
    \sum_{A+B\equiv G~(\moda H^k)}&E(X, D_1^k)(A) E(Y, D_2^k)(B)=\\
    &\left\{\begin{tabular}[l]{l}
          $|H|^kE(X, D^k)$, if $X=Y, D_1=D_2=D$,\\
          $0$, \mbox{ otherwise},
        \end{tabular}
\right.
    \end{split}
  \end{equation}
where the summation ranges over a complete residue system modulo $H^k$. 
\end{lemma}

\begin{proof}
  We first note that $E(G, H) = E(G, a^{-1}H)$, where $a$ is the leading coefficient of $H$. Without loss of generality, we may suppose that $H$ is a monic polynomial, and write $D_1R_1 = H$, $D_2R_2 = H$, and $B=G-A~(\moda H^k)$, then the left side of \eqref{eq:lemma24} is (see \eqref{eq:EGHprop}) that 
\begin{equation}
  \label{eq:lemma24p}
  \begin{split}
&\sum_{A\moda H^k}E(X, D_1^k)(A)E(Y, D_2^k)(G-A)\\
=&E(Y, D_2^k)(G) \sum_{A\moda H^k} E(XR_1^k, H^k)(A)E(YR_2^k, H^k)(-A)\\
=&E(Y, D_2^k)(G) \sum_{A\moda H^k} E(A, H^k)(XR_1^k-YR_2^k).
  \end{split}
\end{equation}
By Lemma~\ref{lem:twoCases}, the inner sum in the above equality is zero, if $XR_1^k-YR_2^k \not \equiv 0 ~(\moda H^k)$, and $|H|^k$, if $XR_1^k\equiv YR_2^k ~(\moda H^k)$. Since $\deg(X)< k \deg(D_1)$ and $\deg(Y)<k \deg(D_2)$, we have $XR_1^k = YR_2^k$, if $XR_1^k\equiv YR_2^k ~(\moda H^k)$. It follows that $XD_2^k = YD_1^k$, and $D_1 = D_2$, $X=Y$, since $(X, D_1^k)_k = (Y, D_2^k)_k = 1$. We complete the proof of this lemma. 
\end{proof}

Next, we give a simple proof of \eqref{eq:ramaPro3} and \eqref{eq:Holder}. In order to prove the Kluyer and H\"{o}lder's formula in the polynomial case, we first make some slight modification of the M\"{o}bius inversion formula on $\fnum_q[x]$. 

\begin{definition}
  A non-zero mapping $\delta$ from $\fnum_q[x]$ to the complex plane is said to be an arithmetical function on $\fnum_q[x]$, if $\delta(aA) = \delta(A)$ for any $a\in \fnum_q^*$ and $A\in \fnum_q[x]$. It is said to be a multiplicative function, if $\delta(AB) = \delta(A)\cdot \delta(B)$, whenever $(A, B)=1$, and a complex multiplicative function, if $\delta(AB) = \delta(A)\cdot\delta(B)$ for any $A, B$ in $\fnum_q[x]$. 
\end{definition}

We start with the following a few important examples of the arithmetical functions on $\fnum_q[x]$. 

{\bf M\"{o}bius function $\mu(H)$}: Let $\mu(0) = 0$, and 
\begin{equation}
  \label{eq:mobius}
  \mu(H) = \left\{
    \begin{tabular}[l]{l}
      1, if  $H\in \fnum_q^*$\\
      0, if there exists a $P$ \mbox{such that} $P^2|H$\\
      $(-1)^t$, if $H=P_1P_2 \cdots P_t$, where $P_j$ are different.
    \end{tabular}\right.
\end{equation}
It is easy to see that $\mu(H)$ is a multiplicative function on $\fnum_q[x]$. Moreover, we have the following identities that
\begin{equation}
  \label{eq:mobiusIden1}
  \sum_{D|H}\mu(D)=\left\{
    \begin{tabular}[l]{l}
      $1$, if $\deg(H)=0$\\
      $0$, if $\deg(H)\geq 1$
    \end{tabular}\right.
\end{equation}
and
\begin{equation}
  \label{eq:mobiusIden2}
  \sum_{D^k|H}\mu(D)=\left\{
    \begin{tabular}[l]{l}
      1, \mbox{ if }$D$ \mbox{is} $k$\mbox{-th power divisors-free}\\
      0, \mbox{ otherwise}
    \end{tabular}\right.
\end{equation}

{\bf Euler totient function $\phi(H)$}: If $H\neq 0$, we define $\phi(H)$ to be the number of polynomials of degree less than $\deg(H)$ that are comprime to $H$, and $\phi(0)=0$. 

{\bf Jordan totient function $\phi_k(H)$}: If $H\neq 0$, $\phi_k(H)$ is the number of polynomials of degree less than $k\cdot \deg(H)$ that have no common $k$-th power divisors other than one with $H^k$. 

Clearly, $\phi_1(H) = \phi(H)$, and $\phi_k(H)$ is a multiplicative function on $\fnum_q[x]$. The following equalities are easy to verify:
\begin{equation}
  \label{eq:JordanIden}
  \sum_{D|H}\phi(D) = |H|, \mbox{ and } \sum_{D|H}\phi_k(D) = |H|^k. 
\end{equation}

{\bf Mangoldt function $\Lambda(H)$}: We define the Mangoldt function $\Lambda(H)$ on $\fnum_q[x]$ by
\begin{equation}
  \label{eq:mangoldt}
  \Lambda(H)=\left\{
    \begin{tabular}[l]{l}
      $\log |P|$, if $H=P^t, t\geq 1$\\
      $0$, otherwise.
    \end{tabular}
\right.
\end{equation}

It is easy to see that 
\begin{equation}
  \label{eq:mangoldtIden}
  \sum_{D|H}\Lambda(D) = \log |D|. 
\end{equation}

\begin{lemma}[M\"{o}bius Inversion Formula] 
\label{lem:mobiusInv}
If $\Delta(H)$ and $\delta(H)$ are two arithmetical functions on $\fnum_q[x]$, then
  \begin{equation}
    \label{eq:lemma25}
    \Delta(H) = \sum_{D|H}\delta(D), \mbox{ for all }H\neq 0, 
  \end{equation}
if and only if 
\begin{equation}
  \label{eq:lemma25-2}
  \delta(H) = \sum_{D|H}\mu(D)\Delta(\frac{H}{D}), \mbox{ for all }H \neq 0. 
\end{equation}
\end{lemma}

\begin{proof}
  The proof is similar to the classical case, from \eqref{eq:mobiusIden1}, we have the lemma immediately. 
\end{proof}

As a direct consequence of Lemma~\ref{lem:mobiusInv}, from \eqref{eq:JordanIden} and \eqref{eq:mangoldtIden}, we have
\begin{equation}
  \label{eq:lemma25p1}
\phi(H)  = \sum_{D|H}\mu(D)|\frac{H}{D}|  = |H|\prod_{P|H}(1-\frac{1}{|P|}),
\end{equation}
\begin{equation}
  \label{eq:lemma25p2}
\phi_k(H)  = \sum_{D|H}\mu(D)|\frac{H}{D}|^k  = |H|^k\prod_{P|H}(1-\frac{1}{|P|^k}),
\end{equation}
and
\begin{equation}
  \label{eq:lemma25p3}
  \Lambda(H) = \sum_{D|H}\log |D| \cdot \mu(\frac{H}{D}). 
\end{equation}

We give a simple proof of \eqref{eq:ramaPro3} and \eqref{eq:Holder} now. The method here is an analogue of Andreson and Appostol \cite{ref:3}. 

\begin{lemma}[Kluyer's formula]
  \label{lem:kluyer}
Let $\eta_k(G, H)$ be the polynomial Ramanujan sums given by \eqref{eq:ramaGen}, and $H \neq 0$, then we have
\begin{equation}
  \label{eq:kluyer}
  \eta_k(G, H) = \sum_{D|H, D^k|G}|D|^k \mu(\frac{H}{D}). 
\end{equation}
\end{lemma}

\begin{proof}
  By \eqref{eq:mobiusIden2}, we have
  \begin{equation}
    \label{eq:lemmaKluP1}
    \begin{split}
    \eta_k(G, H) = & \sum_{R \moda H^k}  E(G, H^k)(R) \sum_{D^k | (R, H^k)_k}\mu(D)\\
=& \sum_{D|H}\mu(D) \sum_{\substack{R \moda H^k\\ D^k|R}}E(G, H^k)(R)\\
=& \sum_{D|H}\mu(D) \sum_{\substack{R \moda H^k\\ D^k|R}}E(R, H^k)(G).
    \end{split}  
  \end{equation}
We write $R=A\cdot D^k$, and note that $E(R, H^k) = E(A, (\frac{H}{D})^k)$ by Lemma \ref{lem:carlitz}. It follows that
  \begin{equation}
    \label{eq:lemmaKluP1}
    \begin{split}
    \eta_k(G, H) = & \sum_{D|H}  \mu(D) \sum_{A \moda (\frac{H}{D})^k}E\left(A, (\frac{H}{D})^k\right)(G)\\
=& \sum_{D|H}\mu(\frac{H}{D}) \sum_{A \moda D^k}E(A, D^k)(G)\\
=& \sum_{\substack{D|H\\D^k|G}}\mu(\frac{H}{D})|D|^k.
    \end{split}  
  \end{equation}
We complete the proof of Lemma \ref{lem:kluyer}. 
\end{proof}

\begin{theorem}[H\"{o}lder formula]
  \label{thm:holder}
For any $G$ and $H$ in $\fnum_q[x]$ and $H\neq 0$, we have
\begin{equation}
  \label{eq:thmHolder}
  \eta_k(G, H) = \phi_k(H)\mu(N)\phi_k^{-1}(N), \mbox{ where }N^k = \frac{H^k}{(G, H^k)_k}. 
\end{equation}
\end{theorem}

\begin{proof}
  Let $(G, H^k)_k = A^k$, then $D|H$ and $D^k|G$ if and only if $D|A$. Let $N=\frac{H}{A}$, then by Lemma \ref{lem:kluyer}, we have
  \begin{equation}
    \label{eq:thmHolderP}
    \begin{split}
    \eta_k(G, H) = & \sum_{D|A}|D|^k  \mu(\frac{NA}{D}) =  \sum_{D|A}\mu(ND)|\frac{A}{D}|^k\\
=& \mu(N)\sum_{\substack{D|A\\(D, N)=1}}\mu(D)|\frac{A}{D}|^k \\
=& \mu(N)|A|^k \prod_{\substack{P|A\\ P\nmid N}}(1-|P|^{-k})\\
=& \frac{\mu(N)|A|^k \prod_{P|A\cdot N}(1-|P|^{-k})}{\prod_{P|N}(1-|P|^{-k})}\\
=& \frac{\mu(N)|H|^k \prod_{P|H}(1-|P|^{-k})}{|N|^k\prod_{P|N}(1-|P|^{-k})}\\
=& \phi_k(H)\mu(N)\phi_k^{-1}(N),
    \end{split}  
\end{equation}
where $N^k = \frac{H^k}{A^k} = \frac{H^k}{(G, H^k)_k}$. We complete the proof of Theorem \ref{thm:holder}.
\end{proof}

By \eqref{eq:thmHolder}, \eqref{eq:Holder} and \eqref{eq:ramaPro3} follow immediately, in particular, we have
\begin{equation}
  \label{eq:holderFollow1}
  \eta_k(G, H) = \phi_k(H), \mbox{ if }H^k | G, 
\end{equation}
and 
\begin{equation}
  \label{eq:holderFollow2}
  \eta_k(G, H) = \mu(H), \mbox{ if }(G, H^k)_k = 1.
\end{equation}

\section{Reciprocity Formula}
\label{sec:reci}

In this section, we give two reciprocity formulas for the polynomial Ramanujan sums. We start with the following lemmas. 

\begin{lemma}
  \label{lem:reci1}
Let $\eta(G, H)$ be the polynomial Ramanujan sum given by \eqref{eq:rama}, and $H$ be a square-free polynomial, then $\mu(H)\eta(G, H)$ is a multiplicative function in variable $G$. 
\end{lemma}

\begin{proof}
  Let $f(G) = \mu(H) \eta(G, H)$, $H=P_1P_2\cdots P_n$, and $G = G_1\cdot G_2$, where $(G_1, G_2) =1$. We set
  \begin{equation}
    \label{eq:lemReci1P1}
    (G, H) = \prod_{i\in \gamma} P_i, \quad (G_1, H) = \prod_{i\in \gamma_1}P_i,\quad (G_2, H) = \prod_{i\in \gamma_2}P_i, 
  \end{equation}
where $\gamma_1 \cup \gamma_2 = \gamma \subset \{1, 2, \ldots, n\}$, and $\gamma_1 \cap \gamma_2$ is an empty set. Then by the H\"{o}lder formula \eqref{eq:thmHolder} ($k=1$) we have
\begin{equation}
  \label{eq:lemReci1P2}
  f(G) = (-1)^{-|\gamma|}\mu^2(H)\prod_{i\in \gamma}\phi(P_i), 
\end{equation}
\begin{equation}
  \label{eq:lemReci1P3}
  f(G_1) = (-1)^{-|\gamma_1|}\mu^2(H)\prod_{i\in \gamma_1}\phi(P_i), 
\end{equation}
and 
\begin{equation}
  \label{eq:lemReci1P4}
  f(G_2) = (-1)^{-|\gamma_2|}\mu^2(H)\prod_{i\in \gamma_2}\phi(P_i).
\end{equation}
Since $|\gamma_1| + |\gamma_2| = |\gamma|$, and $\mu^2(H) = \mu^4(H)$, then $f(G) = f(G_1)\cdot f(G_2)$, and the lemma follows. 
\end{proof}

\begin{lemma}
  \label{lem:reci2}
If $H$ is square-free, then we have
\begin{equation}
  \label{eq:lemReci2}
  \mu(H)\eta(G, H) = \sum_{D|(G, H)}|D|\mu(D). 
\end{equation}
\end{lemma}

\begin{proof}
  Since the both sides of \eqref{eq:lemReci2} are multiplicative in $G$, it suffices to prove that when $G=P^s$. Let $H = P_1P_2\cdots P_n$. If $P \neq P_j$, $1\leq j \leq n$, then the both sides of \eqref{eq:lemReci2} are one, so we may let $P=P_j$. It follows that
  \begin{equation}
    \label{eq:lemReci2P1}
    \mu(H)\eta(G, H) = -\mu^2(H)\phi(P_j) = 1-|P_j|, 
  \end{equation}
and 
\begin{equation}
  \label{eq:lemReci2P2}
  \sum_{D|(G, H)} |D| \mu(D) = 1-|P_j|. 
\end{equation}
The lemma follows at once.
\end{proof}

We note that the sum in \eqref{eq:lemReci2} is symmetric in $G$ and $H$, hence as a direct consequence of Lemma \ref{lem:reci2}, we have

\begin{corollary}\label{cor:reci1}
  Suppose that both $G$ and $H$ are square-free, then
  \begin{equation}
    \label{eq:cor31}
      \mu(H)\eta(G, H) = \mu(G)\eta(H, G).
  \end{equation}
\end{corollary}

\begin{theorem}[The first reciprocity formula]
\label{thm:reci1}
  Let $\bar{H}$ be the largest square-free divisor of $H$, and $H^* = \frac{H}{\bar{H}}$, then
  \begin{equation}
    \label{eq:thm31}
    \frac{\mu(\bar{H})\eta(GH^*, H)}{|H^*|}  = \frac{\mu(\bar{G})\eta(HG^*, G)}{|G^*|}. 
  \end{equation}
\end{theorem}

\begin{proof}
  It is easy to verify the following Hardy's observation that
  \begin{equation}
    \label{eq:thm31P1}
    \eta(G, H)=0, \mbox{ if } H^*\nmid G, \mbox{ and }\phi(H)= |H^*|\phi(\bar{H}). 
  \end{equation}
In order to prove \eqref{eq:thm31}, we first show that
\begin{equation}
  \label{eq:thm31P2}
  \eta(GH^*, H) = |H^*| \eta(G, \bar{H}). 
\end{equation}
By \eqref{eq:thmHolder}($k=1$), one has
\begin{equation*}
  \begin{split}
    \eta(GH^*, H) &= \phi(H)\mu\left(\frac{H}{(GH^*, H)}\right)\phi^{-1}\left(\frac{H}{(GH^*, H)}\right)\\
&= |H^*|\phi(\bar{H})\mu\left(\frac{\bar{H}}{(G, \bar{H})}\right)\phi^{-1}\left(\frac{\bar{H}}{(G, \bar{H})}\right)\\
&= |H^*|\eta(G,\bar{H}),
  \end{split}
\end{equation*}
and \eqref{eq:thm31P2} follows. We note that $\eta(G, \bar{H}) = \eta(\bar{G}, \bar{H})$, it follows by Corollary \ref{cor:reci1} that
\begin{equation*}
  \begin{split}
    \frac{\mu(\bar{H})\eta(GH^*, H)}{|H^*|} &= \mu(\bar{H})\eta(G, \bar{H}) = \mu(\bar{H})\eta(\bar{G}, \bar{H})\\
&=\mu(\bar{G})\eta(\bar{H}, \bar{G}) = \mu(\bar{G})\eta(H, \bar{G}) \\
&=\frac{\mu(\bar{G})\eta(HG^*, G)}{|G^*|}. 
  \end{split}
\end{equation*}
We complete the proof of Theorem~\ref{thm:reci1}.
\end{proof}

The next reciprocity property first appeared in \cite[Theorem 3.8]{ref:20} in the rational case, here we present an analogue in the polynomial case. 

\begin{theorem}[The second reciprocity formula]
  \label{thm:reci2}
If $D_1|H$, and $D_2|H$, then we have
\begin{equation}
  \label{eq:thmReci2}
  \phi(D_2)\eta(\frac{H}{D_2}, D_1) = \phi(D_1)\eta(\frac{H}{D_1}, D_2).
\end{equation}
\end{theorem}

\begin{proof}
  By \eqref{eq:thmHolder}, then
  \begin{equation}
    \label{eq:thmReci2P1}
    \eta(\frac{H}{D_2}, D_1) = \phi(D_1)\mu\left(\frac{D_1}{(\frac{H}{D_2}, D_1)}\right)\phi^{-1}\left(\frac{D_1}{(\frac{H}{D_2},D_1)}\right),
  \end{equation}
and
\begin{equation}
  \label{eq:thmReci2P2}
     \eta(\frac{H}{D_1}, D_2) = \phi(D_2)\mu\left(\frac{D_2}{(\frac{H}{D_1}, D_2)}\right)\phi^{-1}\left(\frac{D_2}{(\frac{H}{D_1},D_2)}\right).
\end{equation}
We write $D=D_1D_2$, then
\begin{equation}
  \label{eq:thmReci2P3}
  \frac{D_1}{(\frac{H}{D_2},D_1)} =   \frac{D}{(H,D)} =   \frac{D_2}{(\frac{H}{D_1},D_2)},
\end{equation}
and the theorem follows at once. 
\end{proof}

The importance of \eqref{eq:thmReci2} lies in the fact that it is equivalent to the H\"{o}lder formula \eqref{eq:thmHolder} ($k=1$). If we take $D_1 = H$, and $D_2 = \frac{H}{(G, H)}$ in \eqref{eq:thmReci2}, and note that
\begin{equation}
  \label{eq:thmReci2Imp}
  \eta(G, H) = \eta((G, H), H), 
\end{equation}
and $\eta(1, H) = \mu(H)$, then
\begin{equation}
  \label{eq:thmReci2Imp2}
  \eta(G, H) = \phi(H)\mu\left(\frac{H}{(G, H)}\right)\phi^{-1}\left(\frac{H}{(G, H)}\right),
\end{equation}
which is the H\"{o}lder formula of $\eta(G, H)$. 

The generalized polynomial Ramanujan sum $\eta_k(G, H)$ seemingly cannot share the first reciprocity formula when $k>1$, since $G$ and $H$ are not symmetric in a generalized version of \eqref{eq:lemReci2}. But, indeed, $\eta_k(G, H)$ shares the second reciprocity formula. 

\begin{theorem}
  \label{thm:reciK}
If $D_1^k|H$ and $D_2^k|H$, then
\begin{equation}
  \label{eq:thmReciK}
  \phi_k(D_2)\eta_k(\frac{H}{D_2^k}, D_1) = \phi_k(D_1)\eta_k(\frac{H}{D_1^k}, D_2).
\end{equation}
\end{theorem}

\begin{proof}
  By \eqref{eq:thmHolder}, we have

  \begin{equation}
    \label{eq:thmReciKP1}
    \eta_k(\frac{H}{D_2^k}, D_1) = \phi_k(D_1)\mu(N_1)\phi_k^{-1}(N_1),
  \end{equation}
and
\begin{equation}
  \label{eq:thmReciKP2}
     \eta_k(\frac{H}{D_1^k}, D_2) = \phi(D_2)\mu(N_2)\phi_k^{-1}(N_2),
\end{equation}
where $N_1^k  = \frac{D_1^k}{(\frac{H}{D_2^k}, D_1^k)_k}$, and $N_2^k  = \frac{D_2^k}{(\frac{H}{D_1^k}, D_2^k)_k}$.

We write $D=D_1D_2$, then
\begin{equation}
  \label{eq:thmReciKP3}
N_1^k = \frac{D^k}{(H,D^k)_k} = N_2^k. 
\end{equation}
It follows that $N_1 = N_2$, and we have Theorem \ref{thm:reciK}.
\end{proof}

As the case of $k=1$, \eqref{eq:thmReciK} is equivalent to \eqref{eq:thmHolder}. If we replace $H$ by $H^k$ in \eqref{eq:thmReciK}, then if $D_1|H$ and $D_2|H$, we have
\begin{equation}
  \label{eq:thmReciKImp}
  \phi_k(D_2)\eta_k(\frac{H^k}{D_2^k}, D_1) = \phi_k(D_1)\eta_k(\frac{H^k}{D_1^k}, D_2).
\end{equation}

As an analogue of the above equality in the rational case, one may see Lemma 1 of \cite{ref:12}. Taking $D_1 = H$, and $D_2=\frac{H}{A}$, where $A^k = (G, H^k)_k$ in \eqref{eq:thmReciKImp}, we note that
\begin{equation}
  \label{eq:thmReciKImp2}
  \eta_k(G, H) = \eta_k(A^k, H), \mbox{ and }\eta_k(1, H) = \mu(H).
\end{equation}
It follows that
\begin{equation*}
  \eta_k(G, H) = \phi_k(H)\mu(N)\phi_k^{-1}(N), 
\end{equation*}
where $N=\frac{H}{A}$, and $N^k=\frac{H^k}{A^k} = \frac{H^k}{(G, H^k)_k}$, we have \eqref{eq:thmHolder} at once. 

To make applications of the first reciprocity formula, one may follow Johnson \cite{ref:24} to consider the $C$-series representations and the $C'$-series representation for the arithmetical function on $\fnum_q[x]$, and show that the two classes of representation are equivalent under certain conditions. Here we consider a special Dirichlet series and derive its values by using Theorem~\ref{thm:reci1}. We set
\begin{equation}
  \label{eq:appThm31}
  \sigma(s, H) = \sum_{\substack{G\in \mathbb{A}\\G~square-free}}\frac{\eta(G, H)}{|G|^s},
\end{equation}
where $\mathbb{A}$ is the set of monic polynomials of $\fnum_q[x]$, and $s$ is a complex variable. If $H$ is square-free, and $\re(s)>1$, we have the following formula
\begin{equation}
  \label{eq:appThm31-2}
  \mu(H)\sigma(s, H) = \zeta_{\mathbb{A}}(s)\zeta^{-1}_{\mathbb{A}}(2s)\prod_{P|H}\frac{1-|P|+|P|^s}{1+|P|^s},
\end{equation}
where $\zeta_{\mathbb{A}}(s)$ is the zeta function on $\fnum_q[x]$ given by \eqref{eq:zeta}. If $\re(s)>1$, then $\zeta_{\mathbb{A}}(s)$ has the following Euler product formula
\begin{equation}
  \label{eq:eulerProd}
  \zeta_{\mathbb{A}}(s) = \prod_{P\in \mathbb{A}}(1-|P|^{-s})^{-1}. 
\end{equation}
To prove \eqref{eq:appThm31-2}, by \eqref{eq:cor31} we have
\begin{equation}
  \label{eq:eulerProdP}
  \mu(H)\sigma(s, H) = \sum_{\substack{G\in \mathbb{A},\\G~square-free}}\frac{\mu(G)\eta(H, G)}{|G|^s},
\end{equation}
where $f(G) = \mu(G)\eta(H, G)$ is a multiplicative function in $G$, so we may make use of Euler product and obtain that
\begin{equation*}
  \label{eq:eulerProdP2}
    \mu(H)\sigma(s, H) = \prod_{P\in \mathbb{A}}\left(1-\frac{\eta(H, P)}{|P|^s}\right).
\end{equation*}
It is easy to see that $\eta(H, P)=\mu(P)$, if $P \nmid H$, and $\eta(H, P) = \phi(P)$, if $P|H$. Then we have
\begin{equation*}
  \begin{split}
    \mu(H)\sigma(s, H) &= \prod_{\substack{P\in \mathbb{A},\\P\nmid H}}(1+\frac{1}{|P|^s})\prod_{P|H}\left(1-\frac{\phi(P)}{|P|^s}\right)\\
&= \zeta_{\mathbb{A}}(s)\zeta_{\mathbb{A}}^{-1}(2s)\prod_{P|H}\frac{1-|P|+|P|^s}{1+|P|^s},
  \end{split}
\end{equation*}
which is \eqref{eq:appThm31-2}.

An application of the second reciprocity formula appears in the next section (See Lemma~\ref{lem:44} below). 

\section{Orthogonality relation}
\label{sec:orth}

In this section we derive some more complicated orthogonal formula and a number of corollaries for the polynomial Ramanujan sums. In the rational case, one may find the analogues in Cohen \cite{ref:11,ref:12}. We begin by proving

\begin{lemma}\label{lem:41}
  If $D_1|H$ and $D_2|H$, then for any $G$ in $\fnum_q[x]$, we have
  \begin{equation}
    \label{eq:lemma41}
    \sum_{A+B\equiv G~(\moda H^k)}\eta_k(A, D_1)\eta_k(B, D_2)=\left\{
        \begin{tabular}[l]{l}
          $|H|^k\eta_k(G, D)$, if $D_1=D_2=D$,\\
          $0$, otherwise,
        \end{tabular}
\right.
  \end{equation}
where the summation extends over a complete residue system modulo $H^k$. 
\end{lemma}

\begin{proof}
  By Lemma~\ref{lem:24}, the left-hand side of \eqref{eq:lemma41} is
  \begin{equation}
    \label{eq:lemma41p1}
    \begin{split}
    \sum_{(X, D_1^k)_k=1}\sum_{(Y, D_2^k)_k=1}&\sum_{A+B\equiv G~(\moda H^k)}E(A, D_1^k)(X)E(B, D_2^k)(Y)  \\
&=|H|^k \sum_{(X, D^k)=1}E(X, D^k)(G) = |H|^k \eta_k(G, D),
    \end{split}    
  \end{equation}
if $D_1=D_2=D$. Otherwise it is zero. We have the lemma immediately. 
\end{proof}

\begin{definition}
  A $k$-reduced residue system modulo $H$ is that $D$ ranges over modulo $H^k$ such that $(D, H^k)_k = 1$. 
\end{definition}

\begin{lemma}\label{lem:42}
  A complete residue system modulo $H^k$ is given by $A=R(\frac{H}{D})^k$, where $D$ ranges over the divisors of $H$, and for each $D$, $R$ ranges over a $k$-reduced residue system modulo $D$. 
\end{lemma}

\begin{proof}
  See Cohen \cite{ref:13}, Lemma 4. 
\end{proof}

The next lemma is a more generalized form of the second reciprocity formula of $\eta_k(G, H)$. 

\begin{lemma}\label{lem:43}
  If $D_1^k|H$, and $D_2^k|H$, then for any $R$ in $\fnum_q[x]$, we have
  \begin{equation}
    \label{eq:lemma43}
    \phi_k(D_2)\eta_k(\frac{RH}{D_2^k}, D_1) = \phi_k(D_1)\eta_k(\frac{RH}{D_1^k}, D_2).
  \end{equation}
\end{lemma}

\begin{proof}
  It follows directly from Theorem~\ref{thm:reciK}. 
\end{proof}

If we replace $H$ by $H^k$ in the above equality, then we have the following corollary. 

\begin{corollary}
  If $D_1|H$, and $D_2|H$, then for any polynomial $R$ in $\fnum_q[x]$ we have (comparing with \cite[Lemma 1]{ref:12}) that
  \begin{equation}
    \label{eq:cor41}
    \phi_k(D_2)\eta_k\left(R(\frac{H}{D_2})^k, D_1\right) = \phi_k(D_1)\eta_k\left(R(\frac{H}{D_1})^k, D_2\right).
  \end{equation}  
\end{corollary}

\begin{lemma}\label{lem:44}
  If $D|H$, $D_1|H$, and $R\in \fnum_q[x]$ such that $(R, D^k)_k=1$, then we have
  \begin{equation}
    \label{eq:lemma44}
   \eta_k\left(R(\frac{H}{D})^k, D_1 \right) =   \eta_k\left((\frac{H}{D})^k, D_1 \right).
  \end{equation}       
\end{lemma}

\begin{proof}
  By \eqref{eq:lemma43}, then
  \begin{equation}
    \label{eq:lemma44p1}
    \eta_k\left(R(\frac{H}{D})^k, D_1 \right) = \frac{\phi_k(D_1)}{\phi_k(D)}\eta_k\left(R(\frac{H}{D_1})^k, D \right).
  \end{equation}
Because of $(R, D^k)_k=1$, it is easy to see by \eqref{eq:thmHolder} that
\begin{equation}
  \label{eq:lemma44p2}
     \eta_k\left(R(\frac{H}{D_1})^k, D \right) =   \eta_k\left((\frac{H}{D_1})^k, D\right). 
\end{equation}
By \eqref{eq:lemma43} once again, we have
\begin{equation}
  \label{eq:lemma44p3}
    \eta_k\left((\frac{H}{D_1})^k, D \right) = \frac{\phi_k(D)}{\phi_k(D_1)}\eta_k\left((\frac{H}{D})^k, D_1 \right),
\end{equation}
and the lemma follows at once. 
\end{proof}

Now we state and prove the main result of this section. 
\begin{theorem}
  \label{thm:41}
  If $D_1|H$, and $D_2|H$, then
  \begin{equation}
    \label{eq:thm41}
    \sum_{D|H}\eta_k\left((\frac{H}{D})^k, D_1\right)\eta_k\left((\frac{H}{D_2})^k, D\right)=\left\{
        \begin{tabular}[l]{l}
          $|H|^k$, if $D_1=D_2$,\\
          $0$, otherwise.
        \end{tabular}
\right.
  \end{equation}
\end{theorem}

\begin{proof}
  By Lemma~\ref{lem:42}, $A\moda H^k$ is given by $A = R(\frac{H}{D})^k$, where $D$ ranges over the divisors of $H$, and for each $D$, $R$ ranges over a $k$-reduced residue system modulo $D$. Therefore, by \eqref{eq:lemma44}, the left side of \eqref{eq:lemma41} may be written as
  \begin{equation}
    \label{eq:thm41p1}
    \sum_{A\moda H^k}\eta_k(A, D_1)\eta_k(G-A, D_2) = \sum_{D|H}\eta_k\left((\frac{H}{D})^k, D_1 \right) \sum_{\substack{R\moda D^k\\(R, D^k)_k=1}}\eta_k\left(G-R(\frac{H}{D})^k, D_2  \right).
  \end{equation}

We consider the inner sum of the right side of \eqref{eq:thm41p1} separately, and denote this sum by $S$, then
\begin{equation}
  \label{eq:thm41p2}
  \begin{split}
    S &= \sum_{\substack{R\moda D^k\\(R, D^k)_k=1}}\sum_{\substack{A\moda D_2^k\\(A, D_2^k)_k=1}}E\left(G-R(\frac{H}{D})^k, D_2^k\right)(A)\\
&=\sum_{\substack{A\moda D^k\\(A, D_2^k)_k=1}}E(G, D_2^k)(A)\sum_{\substack{R\moda D^k\\(R, D^k)_k=1}}E\left(-A(\frac{H}{D})^k, D_2^k\right)(R).
  \end{split}
\end{equation}
We write $D_2 \cdot B = H$, by \eqref{eq:lem22} of Lemma~\ref{lem:carlitz}, then
\begin{equation}
  \label{eq:thm41p3}
  \begin{split}
    E(-A(\frac{H}{D})^k, D_2^k) &= E(-A(\frac{H}{D})^k B^k, H^k)\\
&= E(-A(\frac{H}{D_2})^k (\frac{H}{D})^k, H^k)\\
&= E(-A(\frac{H}{D_2})^k, D^k). 
  \end{split}
\end{equation}
It follows that
\begin{equation}
  \label{eq:thm41p4}
  \begin{split}
    S &=\sum_{\substack{A\moda D_2^k\\(A, D_2^k)_k=1}}E(G, D_2^k)(A)\sum_{\substack{R\moda D^k\\(R, D^k)_k=1}}E\left(-A(\frac{H}{D_2})^k, D^k\right)(R)\\
&=\sum_{\substack{A\moda D_2^k\\(A, D_2^k)_k=1}}E(G, D_2^k)(A) \eta_k\left(-A(\frac{H}{D_2})^k, D\right)\\
&=\sum_{\substack{A\moda D_2^k\\(A, D_2^k)_k=1}}E(G, D_2^k)(A) \eta_k\left((\frac{H}{D_2})^k, D\right)\\
&=\eta_k(G, D_2)\eta_k\left((\frac{H}{D_2})^k, D\right). 
  \end{split}
\end{equation}

By Lemma~\ref{lem:41} and \eqref{eq:thm41p1}, we have
\begin{equation}
  \label{eq:thm41p5}
\eta_k(G, D_2) \sum_{D|H}\eta_k\left((\frac{H}{D})^k, D_1\right)\eta_k\left((\frac{H}{D_2})^k, D\right)=\left\{
        \begin{tabular}[l]{l}
          $|H|^k\eta_k(G, D)$, if $D_1=D_2=D$,\\
          $0$, otherwise.
        \end{tabular}
\right.
  \end{equation}

If we take $G=0$ in the above equality, and note that $\eta_k(0, D_2) = \phi_k(D_2)$ (see \eqref{eq:holderFollow1}), then $\eta_k(0, D_2) \neq 0$, and we have
  \begin{equation}
    \label{eq:thm41p6}
    \sum_{D|H}\eta_k\left((\frac{H}{D})^k, D_1\right)\eta_k\left((\frac{H}{D_2})^k, D\right)=\left\{
        \begin{tabular}[l]{l}
          $|H|^k$, if $D_1=D_2=D$,\\
          $0$, otherwise.
        \end{tabular}
\right.
  \end{equation}
We complete the proof of Theorem~\ref{thm:41}.
\end{proof}

Next, we deduce a number of the arithmetical relations, most of which are the straightforward consequences of \eqref{eq:lemma41} and \eqref{eq:thm41}. 

\begin{corollary}
  \begin{equation}
    \label{eq:cor42}
    \sum_{G\moda H^k}\eta_k(G, H)=\left\{
      \begin{tabular}[l]{l}
        $1$, if $\deg(H)=0$,\\
        $0$, if $\deg(H)\geq 1$.
      \end{tabular}
\right.
  \end{equation}
\end{corollary}

\begin{proof}
  Let $D_1 = H$, and $D_2=1$ in \eqref{eq:lemma41}, we have this corollary at once. 
\end{proof}

\begin{corollary}
  \begin{equation}
    \label{eq:cor43}
        \sum_{D|H}\eta_k(G, D)=\left\{
      \begin{tabular}[l]{l}
        $|H|^k$, if $G\equiv 0~(\moda H^k)$,\\
        $0$, otherwise.
      \end{tabular}
\right.
  \end{equation}
\end{corollary}

\begin{proof}
  Taking $D_1=1$ in Theorem~\ref{thm:41}, we have \eqref{eq:cor43} immediately. 
\end{proof}

\begin{corollary}
  \begin{equation}
    \label{eq:cor44}
        \sum_{D|H}\eta_k\left((\frac{H}{D})^k, H\right)\eta_k(G, D)=\left\{
      \begin{tabular}[l]{l}
        $|H|^k$, if $(G, H^k)_k=1$,\\
        $0$, otherwise. 
      \end{tabular}
\right.
  \end{equation}  
\end{corollary}

\begin{proof}
  Let $D_1=H$ in \eqref{eq:thm41}, and by \eqref{eq:lemma44} we have this corollary. 
\end{proof}

\begin{corollary}
  If $D_1|H$, then we have
  \begin{equation}
    \label{eq:cor45}
        \sum_{D|H}\eta_k\left((\frac{H}{D})^k, D_1\right)\phi_k(D)=\left\{
      \begin{tabular}[l]{l}
        $|H|^k$, if $D_1=1$,\\
        $0$, otherwise. 
      \end{tabular}
\right.
  \end{equation}  
\end{corollary}

\begin{proof}
  If we take $D_2=1$ in \eqref{eq:thm41}, and note that $\eta_k(H^k, D)=\phi_k(D)$, then \eqref{eq:cor45} follows immediately. 
\end{proof}

Using \eqref{eq:lemma43} with $R=1$, we may reformulate Theorem~\ref{thm:41} as follows. 

\begin{corollary}
  If $D_1|H$ and $D_2|H$, then
  \begin{equation}
    \label{eq:corxx}
    \sum_{D|H}\frac{1}{\phi_k(D)}\eta_k\left((\frac{H}{D_1})^k, D\right)\eta_k\left((\frac{H}{D_2})^k, D\right)=\left\{
        \begin{tabular}[l]{l}
          $\frac{|H|^k}{\phi_k(D')}$, if $D_1=D_2=D'$,\\
          $0$, otherwise.
        \end{tabular}
\right.
  \end{equation}
\end{corollary}

In particular, for any $D_1|H$, we have
\begin{equation}
  \label{eq:reformThm41}
  \sum_{D|H}\frac{1}{\phi_k(D)}\eta_k\left((\frac{H}{D_1})^k, D\right)^2 = \frac{|H|^k}{\phi_k(D_1)}.
\end{equation}
If we make use of the H\"{o}lder formula in the above equality, it yields

\begin{corollary}
  If $D_1D_1'=H$, then
  \begin{equation}
    \label{eq:cor47}
    \sum_{\substack{D|H\\N(D_1', D)=D}}\phi_k(D)\left(\frac{\mu(N)}{\phi_k(N)}\right)^2 = \frac{|H|^k}{\phi_k(D_1)}. 
  \end{equation}
\end{corollary}

The special cases of $D_1=1$ and $D_1=H$ lead to the following corollaries respectively
\begin{corollary}[see \eqref{eq:JordanIden}]
  \begin{equation}
    \label{eq:cor48}
    \sum_{D|H}\phi_k(D) = |H|^k
  \end{equation}
\end{corollary}
and 
\begin{corollary}
  \begin{equation}
    \label{eq:cor49}
    \sum_{D|H}\frac{|\mu(D)|}{\phi_k(D)} = \frac{|H|^k}{\phi_k(H)}. 
  \end{equation}
\end{corollary}

In fact, we have the following more generalized conclusions. 

\begin{lemma}
  If $R|H$, then we have 
  \begin{equation}
    \label{eq:lemma45}
    \alpha_k(R, H)=\sum_{\substack{D|\frac{H}{R}\\(D, R)=1}}\frac{|\mu(D)|}{\phi_k(D)} = \frac{\phi_k(R)}{\phi_k(H)}\left|\frac{H}{R}\right|^k.
  \end{equation}
\end{lemma}

\begin{proof}
  The left side of \eqref{eq:lemma45} has the following product expression
  \begin{equation*}
    \begin{split}
      \alpha_k(R, H)&= \prod_{\substack{P|\frac{H}{R}\\P\nmid R}}(1+\frac{1}{\phi_k(P)})\\
&= \prod_{P|H}(1+\frac{1}{\phi_k(P)})\prod_{P|R}(1+\frac{1}{\phi_k(P)})^{-1}\\
&= \frac{\phi_k(R)}{\phi_k(H)}\left| \frac{H}{R}\right|^k,
    \end{split}
  \end{equation*}
and the lemma follows. 
\end{proof}

Next, we generalize the Jordan totient function $\phi_k(H)$ to $\phi_s(H)$, where $s$ is a complex number variable. We define (see \eqref{eq:AppCohen3} in the rational case)
\begin{equation}
  \label{eq:JordanGen}
  \phi_s(H) = \sum_{D|H} |D|^s \mu(\frac{H}{D}) = |H|^s\prod_{P|H}(1-\frac{1}{|P|^s}).
\end{equation}

\begin{lemma}\label{lemma:46}
  If $s$ is an arbitrary complex number, then
  \begin{equation}
    \label{eq:lemma46}
    \sum_{D|H}\eta_k\left((\frac{H}{D})^k, H\right)\phi_{ks}(D) = |H|^{ks}\phi_{k(1-s)}(H). 
  \end{equation}
\end{lemma}

\begin{proof}
  By \eqref{eq:thmHolder} and \eqref{eq:JordanGen}, we have
  \begin{equation}
    \label{eq:lemma46p1}
    \begin{split}
    \sum_{D|H}\eta_k\left((\frac{H}{D})^k, H\right)\phi_{ks}(D) &= \phi_k(H)\sum_{D|H}\frac{\mu(D)}{\phi_k(D)}\phi_{ks}(D)\\
&=\phi_k(H)\sum_{D|H}\frac{\mu(D)}{\phi_k(D)}\sum_{R|D}|R|^{ks}\mu(\frac{D}{R})\\
&=\phi_k(H)\sum_{D|H}\sum_{\substack{R\cdot E = D\\(R, E)=1}}\frac{\mu^2(E)|R|^{ks}\mu(R)}{\phi_k(R)\phi_k(E)}\\
&=\phi_k(H)\sum_{R|H}\frac{|R|^{ks}\mu(R)}{\phi_k(R)}\sum_{\substack{E|\frac{H}{R}\\(E, R)=1}}\frac{|\mu(E)|}{\phi_k(E)}\\
& =|H|^{ks}\phi_{k(1-s)}(H),
    \end{split}
  \end{equation}
and the lemma follows. 
\end{proof}

We may obtain more orthogonal relation formulas from \eqref{eq:lemma41} and \eqref{eq:thm41}. For example, if we take $G=0$ in \eqref{eq:lemma41}, and note that $\eta_k(-B, D_2) = \eta_k(B, D_2)$, if follows that
\begin{corollary}
  If $D_1|H$ and $D_2|H$, then
  \begin{equation}
    \label{eq:cor410}
    \sum_{A\moda H^k}\eta_k(A, ,D_1)\eta_k(A, D_2)=\left\{
        \begin{tabular}[l]{l}
          $|H|^k\phi_k(D)$, if $D_1=D_2=D$,\\
          $0$, otherwise.
        \end{tabular}
\right.
  \end{equation}
In particular, if $D|H$, then
\begin{equation}
  \label{eq:cor410-2}
  \sum_{A\moda H^k}\eta_k(A, D)^2 = |H|^k \phi_k(D). 
\end{equation}
\end{corollary}

\section{The series $\delta_k(s, G)$}
\label{sec:series}

In this section, we prove Theorem~\ref{thm:11}. Recalling the Dirichlet series $\delta_k(s, G)$ is given by (see \eqref{eq:Diri})
\begin{equation}
  \label{eq:DiriDef}
  \delta_k(s, G) = \sum_{H\in \mathbb{A}}\frac{\eta_k(G, H)}{|H|^s}.
\end{equation}

\begin{lemma}\label{lem:51}
  If $\re(s)>k+1$, $H^k\nmid G$, then we have
  \begin{equation}
    \label{eq:lemma51}
    \delta_k(s, G) = (1-q^{1-s})\sum_{D^k|G}|D|^{k-s}.
  \end{equation}
\end{lemma}

\begin{proof}
  By \eqref{eq:kluyer}, we have
  \begin{equation}
    \label{eq:lemma51p1}
    \begin{split}
      \delta_k(s, G)&= \sum_{H\in \mathbb{A}}\frac{1}{|H|^s}\sum_{D|H, D^k|G}|D|^k \mu(\frac{H}{D})\\
&= \sum_{D^k|G}|D|^k\sum_{\substack{H\in \mathbb{A}\\D|H}}\frac{\mu(\frac{H}{D})}{|H|^s}\\
&= \sum_{D^k|G}|D|^{k-s}\sum_{H \in \mathbb{A}}\frac{\mu(H)}{|H|^s}.
    \end{split}
  \end{equation}
If $\re(s)>1$, we have (see \cite[Proposition 2.6]{ref:43})
\begin{equation}
  \label{eq:lemma51p2}
  \sum_{H\in \mathbb{A}} \frac{\mu(H)}{|H|^s} = \zeta_{\mathbb{A}}^{-1}(s) = (1-q^{1-s}),
\end{equation}
and the lemma follow at once. 
\end{proof}

{\bf Proof of Theorem \ref{thm:11}}. We first show that $\delta_k(s, G)$ is an entire function on the whole complex plane. For any $n\geq 0$, we let
\begin{equation}
  \label{eq:thm11p1}
  A(n) = \sum_{\substack{H\in \mathbb{A}\\\deg(H)=n}}\eta_k(G, H). 
\end{equation}
By definition \eqref{eq:Diri}, we have
\begin{equation}
  \label{eq:thm11p2}
  \delta_k(s, G)=\sum_{n=0}^{+\infty}A(n)q^{-ns}=\sum_{n=0}^{+\infty}A(n)u^n,
\end{equation}
where $u=q^{-s}$. On the other hands, if $G\neq 0$, and $\re(s)>k+1$, by Lemma \ref{lem:51},
\begin{equation}
  \label{eq:thm11p3}
  \begin{split}
      \delta_k(s, G) &= (1-q^{1-s})\sum_{D^k|G}q^{dk}q^{-ds} \\
&= (1-qu)\sum_{D^k|G}q^{dk}u^d\\
&= \sum_{D^k|G}q^{dk}u^d - \sum_{D^k|G}q^{dk+1}u^{d+1},
  \end{split}
\end{equation}
where $d=\deg(D)$. We set
\begin{equation}
  \label{eq:thm11p4}
  \gamma(G, n) = \sum_{\substack{D^k|G\\\deg(D)=n}}1. 
\end{equation}
By \eqref{eq:thm11p3}, we have
\begin{equation}
  \label{eq:thm11p5}
  \begin{split}
      \delta_k(s, G) &= \sum_{n=0}^{+\infty}q^{nk}\gamma(G, n)u^n - \sum_{n=0}^{+\infty}q^{nk+1}\gamma(G, n)u^{n+1}\\
&=\sum_{n=0}^{+\infty}q^{nk}\gamma(G, n)u^n - \sum_{n=1}^{+\infty}q^{(n-1)k+1}\gamma(G, n-1)u^{n}\\
&=\sum_{n=0}^{+\infty}q^{nk}\left(\gamma(G, n) -q^{1-k}\gamma(G, n-1)\right)u^n. 
  \end{split}
\end{equation}
Comparing the coefficients of $u^n$ of \eqref{eq:thm11p2} and \eqref{eq:thm11p5}, we have
\begin{equation}
  \label{eq:thm11p6}
  A(n) = q^{nk}\left(\gamma(G, n)-q^{1-k}\gamma(G, n-1) \right).
\end{equation}
By the definition of $\gamma(G, n)$, if $n-1 > \frac{\deg(G)}{k}$, it is easy to see that $\gamma(G, n) = \gamma(G, n-1) = 0$, and it follows that
\begin{equation}
  \label{eq:thm11p7}
  A(n)=0, \mbox{ whenever }n-1 > \frac{\deg(G)}{k}. 
\end{equation}
We note that the definition of $A(n)$ is independent on the choice of $s$, thus for any $s$, we have
\begin{equation}
  \label{eq:thm11p8}
  \delta_k(s, G) = \sum_{0\leq n \leq \frac{\deg(G)}{2}+1}A(n)u^n = \sum_{\substack{H\in \mathbb{A}\\\deg(H)\leq \frac{\deg(G)}{2}+1}}\frac{\eta_k(G, H)}{|H|^s}, 
\end{equation}
which indicates that $\delta_k(s, G)$ is, indeed, a finite summand, therefore, $\delta_k(s, G)$ is an entire function, and on the whole complex plane we have
\begin{equation}
  \label{eq:thm11p9}
  \delta_k(s, G) = (1-q^{1-s})\sum_{D^k|G}|D|^{k-s}. 
\end{equation}
In particular, if $s=1$, then we have
\begin{equation*}
  \delta_k(1, G) =\sum_{H\in \mathbb{A}}\frac{\eta_k(G, H)}{|H|} =0.
\end{equation*}

To complete the proof of Theorem \ref{thm:11}, next we show that square mean values estimate \eqref{eq:thm11-4}. If $c$ and $T$ are two given real numbers and $T>0$, by \eqref{eq:thm11p9}, we have
\begin{equation}
  \label{eq:thm11p11}
    \int_{-T}^{T}|\delta_k(c+it, G)|^2 dt = \sum_{D_1^k|G, D_2^k|G}q^{(d_1+d_2)(k-c)}\int_{-T}^{T}(1-q^{1-c-it})(1-q^{1-c+it})q^{(d_2-d_1)it} dt,
\end{equation}
where $d_1 = \deg(D_1)$ and $d_2 = \deg(D_2)$. We denote the inner integral of \eqref{eq:thm11p11} by $S(d_1, d_2)$, 
\begin{equation}
  \label{eq:thm11p12}
  S(d_1, d_2) = \int_{-T}^T (1-q^{1-c-it})(1-q^{1-c+it})q^{(d_2-d_1)it} dt.
\end{equation}
Making substitution $u=q^{it}$, it follows that
\begin{equation}
  \label{eq:thm11p13}
  S(d_1, d_2) = \frac{1}{i\log q}\int_{q^{-iT}}^{q^{iT}} (1+q^{2(1-c)}-q^{1-c}u^{-1} - q^{1-c}u)u^{d_2-d_1-1}du. 
\end{equation}
Let $n=d_2-d_1-1$, if $n\neq -2, -1, 0$, then
\begin{equation*}
  \label{eq:thm11p14}
  \begin{split}
    S(d_1, d_2) =& \frac{2}{\log q}\left[\frac{1+q^{2(1-c)}}{n+1} (q^{i(n+1)T}-q^{-i(n+1)T})\right.\\
   &- \left.\frac{q^{1-c}}{n+2}(q^{i(n+2)T} - q^{-i(n+2)T}) - \frac{q^{1-c}}{n}(q^{inT} - q^{-inT})\right],
  \end{split}
\end{equation*}
which yields the following estimate
\begin{equation}
  \label{eq:thm11p15}
  |S(d_1, d_2)| \leq \frac{4}{\log q}(\frac{1+q^{2(1-c)}}{n+1} + \frac{q^{1-c}}{n+2} + \frac{q^{1-c}}{n}). 
\end{equation}

If $n=d_2-d_1-1=0$, then $d_2 = 1+d_1$, by \eqref{eq:thm11p13}, we have
\begin{equation}
  \label{eq:thm11p16}
  \begin{split}
  S(d_1, d_2) &= \frac{1}{i\log q}\left[(1+q^{2(1-c)})(q^{iT}-q^{-iT})-2iTq^{1-c}\log q - \frac{q^{1-c}}{2}(q^{2iT}-q^{-2iT})  \right]\\
&= -2Tq^{1-c}+O(1).
  \end{split}
\end{equation}

If $n=d_2-d_1-1=-1$, then $d_2=d_1$, and
\begin{equation}
  \label{eq:thm11p17}
  \begin{split}
  S(d_1, d_2) &= \frac{1}{i\log q}\left[2iT(1+q^{2(1-c)})\log q - 2q^{1-c}(q^{iT} -q^{-iT})\right]\\
&= -2T(1+q^{2(1-c)})+O(1).  
  \end{split}
\end{equation}

If $n=d_2-d_1-1=-2$, we also have
\begin{equation}
  \label{eq:thm11p18}
S(d_1, d_2) = -2Tq^{1-c}+O(1).
\end{equation}

Putting the above equalities together, by \eqref{eq:thm11p11} we have
\begin{equation*}
  \begin{split}
    \frac{1}{2T}\int_{-T}^{T}|\delta_k(c+it, G)|^2 dt &= (1+q^{2(1-c)})\sum_{D^k|G}q^{2d(k-1)} 2q^{1-c} \sum_{\substack{D_1^k|G, D_2^k|G\\d_1=d_2+1}}q^{(d_1+d_2)(k-c)}+O(\frac{1}{T})\\
&= (1+q^{2(1-c)})\sum_{D^k|G}|D|^{2(k-c)} - 2q^{1-k} \sum_{\substack{D_1^k|G, D_2^k|G\\d_1=d_2+1}}|D_1|^{2(k-c)}+O(\frac{1}{T}).
  \end{split}
\end{equation*}

Let 
\begin{equation}
  \label{eq:thm11app}
  \sigma_0(x, G) = \sum_{D^k|G}|D|^x, \mbox{ and } \sigma_1(x, G) = \sum_{\substack{D_1^k|G, D_2^k|G\\d_1=d_2+1}}|D_1|^x, 
\end{equation}
then we finally obtain
\begin{equation*}
    \frac{1}{2T}\int_{-T}^{T}|\delta_k(c+it, G)|^2 dt = (1+q^{2(1-c)})\sigma_0(2(k-c)) - 2q^{1-k}\sigma_1(2(k-c))+O(\frac{1}{T}).
\end{equation*}
We complete the proof of Theorem \ref{thm:11}. 

\section{The series $\tau_k(s, H)$}
\label{sec:tau}

In this section we prove Theorem~\ref{thm:12}. Recalling the Dirichlet series $\tau_k(s, H)$ is given by (see \eqref{eq:Diri2}) that
\begin{equation}
  \label{eq:tauDef}
  \tau_k(s, H)  = \sum_{G \in \mathbb{A}}\frac{\eta_k(G, H)}{|G|^s}.
\end{equation}
This series is more complicated than $\delta_k(s, G)$, because that $\eta_k(G, H)$ is not multiplicative in $G$, thus we cannot make use of Euler product directly. Our treatment begins with the following auxiliary series
\begin{equation}
  \label{eq:tauAux}
  \tau_k'(s, H) = \sum_{G\in \mathbb{A}}\frac{\chi(G)}{|G|^s},
\end{equation}
where $\chi(G) = 1$, if $(G, H^k)_k = 1$, and $\chi(G)=0$, if $(G, H^k)_k>1$. Since $\chi(G)$ is a multiplicative function in $G$, so we have the following equality by using Euler product that (if $\re(s)>1$)
\begin{equation}
  \label{eq:tauAuxEul}
  \begin{split}
  \tau_k'(s, H) &= \prod_{D\in \mathbb{A}}(1-|P|^{-s})^{-1} \prod_{P|H}(1-|P|^{-ks})\\  
   &= \zeta_{\mathbb{A}}(s)\phi_{ks}(H)|H|^{-ks}. 
  \end{split}  
\end{equation}

\begin{lemma}\label{lem:61}
  If $H$ is positive polynomial and $\re(s)>k+1$, then
  \begin{equation}
    \label{eq:lemma61}
    \tau_k(s, H) = (1-q^{1-s})^{-1}\phi_{k(1-s)}(H). 
  \end{equation}
\end{lemma}

\begin{proof}
  Let $A^k = (G, H^k)_k$. By \eqref{eq:thmReciKImp2}, we have $\eta_k(G, H) = \eta_k(A^k, H)$, it follows that
  \begin{equation*}
    \label{eq:lemma61}
    \begin{split}
    \tau_k(s, H) &= \sum_{A|H}\eta_k(A^k, H) \sum_{\substack{G\in \mathbb{A}\\A^k = (G, H^k)_k}}\frac{1}{|G|^s}  \\
&= \sum_{A|H}\frac{\eta_k(A^k, H)}{|A|^{sk}} \sum_{\substack{G_1\in \mathbb{A}\\(G_1, (\frac{H}{A})^k)_k=1}}\frac{1}{|G_1|^s}  \\
&= \zeta_{\mathbb{A}}(s)|H|^{-ks}\sum_{A|H}\eta_k(A^k, H)\phi_{ks}(\frac{H}{A}).
    \end{split}
  \end{equation*}
By Lemma~\ref{lemma:46}, we have
\begin{equation*}
  \label{eq:lemma61p}
  \tau_k(s, H) = \zeta_{\mathbb{A}}(s) \phi_{k(1-s)}(H) = (1-q^{1-s})^{-1}\phi_{k(1-s)}(H). 
\end{equation*}
We complete the proof of Lemma~\ref{lem:61}.
\end{proof}

{\bf Proof of Theorem~\ref{thm:12}}. We now return to the proof of Theorem~\ref{thm:12}. For any integer $n\geq 0$, we let
\begin{equation}
  \label{eq:thm12p1}
  B(n) = \sum_{\substack{G\in \mathbb{A}\\\deg(G)=n}} \eta_k(G, H). 
\end{equation}
By definition \eqref{eq:Diri2}, we have
\begin{equation}
  \label{eq:thm12p2}
  \tau_k(s, H) = \sum_{n=0}^{+\infty}B(n)q^{-ns} = \sum_{n=0}^{+\infty}B(n)u^n,
\end{equation}
where $u=q^{-s}$. If $\re(s)>k+1$, by Lemma~\ref{lem:61}
\begin{equation}
  \label{eq:thm12p3}
  \begin{split}
    \tau_k(s, H) &= (1-q^{1-s})^{-1}\phi_{k(1-s)}(H) = (1-q^{1-s})^{-1}\sum_{D|H}|D|^{k(1-s)}\mu(\frac{H}{D})\\
&= (1-qu)^{-1}\sum_{D|H}\mu(\frac{H}{D})q^{dk}u^{dk},
  \end{split}
\end{equation}
where $u=q^{-s}$ and $d=\deg(D)$. Because of $|qu| < 1$, then $(1-qu)^{-1}$ has a geometric series expression, we have
\begin{equation}
  \label{eq:thm12p4}
  \begin{split}
    \tau_k(s, H) &= \sum_{D|H}\mu(\frac{H}{D})\sum_{n=0}^{+\infty}q^{n+dk}u^{n+dk}\\
&= \sum_{D|H}\mu(\frac{H}{D}) \sum_{n=dk}^{+\infty}q^{n}u^{n}.
  \end{split}
\end{equation}

We set
\begin{equation}
  \label{eq:thm12p5}
  J_k(D,n)=\left\{
    \begin{tabular}[l]{l}
      $1$, if $n \geq k \deg(D)$,\\
      $0$, if $n < k \deg(D)$.
    \end{tabular}
\right.
\end{equation}
It follows from \eqref{eq:thm12p4} that
\begin{equation}
  \label{eq:thm12p5}
  \tau_k(s, H) = \sum_{n=0}^{+\infty}\left(\sum_{D|H}\mu(\frac{H}{D})J_k(D, n)  \right) q^n u^n.
\end{equation}
Comparing coefficients of $u^n$ of \eqref{eq:thm12p2} and \eqref{eq:thm12p5}, we have
\begin{equation}
  \label{eq:thm12p6}
  B(n) = \sum_{D|H}\mu(\frac{H}{D})J_k(D, n)\cdot q^n.
\end{equation}

If $n\geq k \log(H)$, then $J_k(D, n)=1$ for all of $D$ that $D|H$, hence
\begin{equation}
  \label{eq:thm12p7}
  B(n) = q^n \sum_{D|H}\mu(\frac{H}{D}) = 0.
\end{equation}
The last equality of \eqref{eq:thm12p7} follows from \eqref{eq:mobiusIden1} and $\deg(H)\geq 1$. We note that $B(n)$ is independent on the choice of complex number $s$ by the definition of $B(n)$, therefore, for any complex number $s$ we have $B(n)=0$, whenever $n\geq k\cdot \deg(H)$. It follows from \eqref{eq:thm12p2} that
\begin{equation}
  \label{eq:thm12p8}
  \tau_k(s, H) = \sum_{n \leq k\deg(H)}B(n)u^n = \sum_{\substack{G\in \mathbb{A}\\\deg(G)\leq k \cdot \deg(H)}}\frac{\eta_k(G, H)}{|G|^s},
\end{equation}
which indicates $\tau_k(s, H)$ is, indeed, a finite summand, thus $\tau_k(s, H)$ is an entire function on the whole complex plane. Moreover, by the continued principle, on the whole complex plane we have
\begin{equation}
  \label{eq:thm12p9}
  \tau_k(s, H) = (1-q^{1-s})^{-1}\phi_{k(1-s)}(H). 
\end{equation}
In particular, $s=1$ is a vanished pole of $\tau_k(s, H)$, and we have
\begin{equation}
  \label{eq:thm12p10}
  \begin{split}
    \tau_k(1, H) &= \lim_{s\rightarrow 1}(1-q^{1-s})^{-1}\phi_{k(1-s)}(H)\\
&= -\frac{k\sum_{D|H}\log |D| \mu(\frac{H}{D})}{\log q} = \frac{-k\Lambda(H)}{\log q}, 
  \end{split}
\end{equation}
which is the equality \eqref{eq:thm12-3}, and is an analogue of Ramanujan's identity \eqref{eq:famousIden}.

In order to complete the proof of Theorem \ref{thm:12}, it remains to prove the square mean value estimate. If $T$ and $c$ are any real numbers that $T>0$, $c\neq 1$, by \eqref{eq:thm12p9} we have

\begin{equation}
  \label{eq:thm12p11}
  \int_{-T}^T |\tau_k(c+it, H)|^2 dt = \sum_{D_1|H, D_2|H} \mu(\frac{H}{D_1})\mu(\frac{H}{D_2})|D_1D_2|^{1-c} \int_{-T}^T \frac{q^{(d_2-d_1)it} dt}{(1-q^{1-c-it})(1-q^{1-c+it})}.
\end{equation}

We denote 
\begin{equation}
  \label{eq:thm12p12}
  S_1(d_1, d_2) = \int_{-T}^T \frac{q^{(d_2-d_1)it} dt}{(1-q^{1-c-it})(1-q^{1-c+it})},
\end{equation}
and make the substitution of $u=-t$, it follows that
\begin{equation}
  \label{eq:thm12p13}
  S_1(d_1, d_2) = \int_{-T}^T \frac{q^{(d_1-d_2)it} dt}{(1-q^{1-c-it})(1-q^{1-c+it})},
\end{equation}
which shows that $d_1$ and $d_2$ are symmetric in \eqref{eq:thm12p11}. Therefore we may suppose that $d_2 \geq d_1$, and make the substitution $u=q^{it}$ in \eqref{eq:thm12p12}, then
\begin{equation}
  \label{eq:thm12p13}
   S_1(d_1, d_2) =\frac{i}{q^{1-c}\log q} \int_{q^{-iT}}^{q^{iT}} \frac{u^{k(d_2-d_1)}du}{(u-q^{1-c})(u-q^{c-1})}.
\end{equation}
If $d_1 = d_2$, then we have
\begin{equation}
  \label{eq:thm12p14}
   S_1(d_1, d_2) =\frac{i}{(q^{2(1-c)}-1)\log q} [\log (q^{iT}-q^{1-c}) - \log(q^{-iT}-q^{1-c}) - \log(q^{iT}-q^{c-1}) + \log(q^{-iT}-q^{c-1})].
\end{equation}
We note that $1-c$ and $c-1$ are symmetric in the above, first we let $c>1$, then
\begin{equation}
  \label{eq:thm12p15}
  \log(q^{iT}-q^{1-c}) - \log(q^{-iT}-q^{1-c}) = 2iT\log q + \log(1-q^{-iT}q^{1-c}) - \log(1-q^{iT}q^{1-c}). 
\end{equation}
Let $z=q^{1-c}q^{\pm iT}$, then $|z| =q^{1-c}<1$ for $c>1$, and we have the following power series expansion that
\begin{equation*}
  -\log(1-z) = z+\frac{1}{2}z^2 + \frac{1}{3}z^3 + \cdots.
\end{equation*}
It follows that
\begin{equation}
  \label{eq:thm12p16}
  |\log(1-z)| \leq |z|+ \frac{1}{2}|z|^2 + \cdots \leq \log\frac{1}{1-|z|} = \log\frac{1}{1-q^{1-c}},
\end{equation}
and by \eqref{eq:thm12p15} we have
\begin{equation}
  \label{eq:thm12p17}
  \log(q^{iT}-q^{1-c}) - \log(q^{-iT}-q^{1-c}) = 2iT\log q+O(1). 
\end{equation}
The remaining part of \eqref{eq:thm12p14} is 
\begin{equation*}
  \begin{split}
    \log(q^{iT}-q^{c-1}) - \log(q^{iT}-q^{c-1}) &= \log \frac{1-q^{1-c}q^{iT}}{1-q^{1-c}q^{-iT}}\\
&= \log(1-q^{1-c}q^{iT}) - \log(1-q^{1-c}q^{-iT}).
  \end{split}
\end{equation*}
By \eqref{eq:thm12p16} we have
\begin{equation}
  \label{eq:thm12p18}
  |\log(1-q^{1-c}q^{iT}) - \log(1-q^{1-c}q^{-iT})| \leq 2\log\frac{1}{1-q^{1-c}}. 
\end{equation}
Hence, if $d_2=d_1$, and $c>1$, we obtain
\begin{equation}
  \label{eq:thm12p19}
  S_1(d_1, d_2) = \frac{2T}{1-q^{2(1-c)}} + O(1). 
\end{equation}
If $d_1=d_2$, and $c<1$, the same method yields the following estimate
\begin{equation}
  \label{eq:thm12p20}
  S_1(d_1, d_2) = \frac{2T}{q^{2(1-c)}-1}+O(1).
\end{equation}
Therefore, if $d_1=d_2$, we have
\begin{equation}
  \label{eq:thm12p21}
  S_1(d_1, d_2) = \frac{2T}{|1-q^{2(1-c)}|}+O(1).
\end{equation}
Next, we consider $d_2>d_1$, and let $n=(d_2-d_1)k$. By \eqref{eq:thm12p13}, then
\begin{equation}
  \label{eq:thm12p22}
  S_1(d_1, d_2) = \frac{i}{(q^{2(1-c)}-1)\log q}\left[\int_{q^{-iT}-q^{1-c}}^{q^{iT}-q^{1-c}}\frac{(u+q^{1-c})^n}{u}du - \int_{q^{-iT}-q^{c-1}}^{q^{iT}-q^{c-1}}\frac{(u+q^{c-1})^n}{u}du   \right].
\end{equation}
We write
\begin{equation}
  \label{eq:thm12p23}
  (u+q^{\pm(1-c)})^n = \sum_{j=0}^n\binom{n}{j}u^j q^{\pm(1-c)(n-j)}.
\end{equation}
If $j\neq 0$, then it is easy to verify that the inner integral in \eqref{eq:thm12p22} is $O(1)$. if $j=0$, there is a similar argument like the case of $d_1=d_2$, which yields
\begin{equation}
  \label{eq:thm12p24}
  S_1(d_1, d_2) = \frac{q^{k(d_2-d_1)(1-c)}}{|1-q^{2(1-c)}|}2T + O(1).
\end{equation}
By \eqref{eq:thm12p11}, we finally obtain

\begin{equation}
  \label{eq:thm12p25}
  \begin{split}
    \frac{1}{2T}\int_{-T}^T |\tau_k(c+it, H)|^2 dt =& \frac{1}{|1-q^{2(1-c)}|}\sum_{D|H}\mu^2(\frac{H}{D})|D|^{2(1-c)}\\
&+ \frac{2}{|1-q^{2(1-c)}|}\sum_{\substack{D_1|H, D_2|H\\ \deg(D_1)> \deg(D_2)}} \left(\mu(\frac{H}{D_1})\mu(\frac{H}{D_2})\right.\\
&\left.|D_1|^{(k+2)(1-c)}|D_2|^{(2-k)(1-c)}\right) +O(\frac{1}{T}). 
  \end{split}
\end{equation}
We complete the proof of Theorem \ref{thm:12}. 

\section{Davenport--Hasse type formula}
\label{sec:daven}

The polynomial Ramanujan sum $\eta(G, H)$ essentially is a special Gauss sums on $\fnum_q[x]$. Let $\chi$ be a multiplicative character modulo $H$ on $\fnum_q[x]$, and $\psi_G= E(G, H)$ be the additive character modulo $H$ given by \eqref{eq:EGH}, the Gauss susm $G(\chi, \psi_G)$ modulo $H$ on $\fnum_q[x]$ is defined by
\begin{equation}
  \label{eq:gaussSum}
  G(\chi, \psi_G) = \sum_{D\moda H}\chi(D)\psi_G(D), 
\end{equation}
where $D$ extends over a complete residue system modulo $H$ in $\fnum_q[x]$. Let $\chi_0$ be the principal multiplicative character, it is $\chi_0(D)=1$ if $(D, H)=1$, and $\chi_0(D)=0$ if $(D, H)>1$, then we see that $\eta(G, H)=G(\chi_0, \psi_G)$. 

In an upcoming paper \cite{ref:49}, we presented an analogue of Davenport--Hasse's theorem for the polynomial Gauss sums (see \cite[Theorem~1.3]{ref:49}). To state this result, let $\fnum_{q^n}/\fnum_q$ be a finite extension over $\fnum_q$ of degree $n$, $\tr(a)$ and $N(a)$ be the trace map and norm from $\fnum_{q^n}$ to $\fnum_q$ respectively, 
\begin{equation}
  \label{eq:trN}
  \tr(a) = \sum_{i=1}^n \sigma^i(a), \mbox{ and } \norm(a) = \prod_{i=1}^n \sigma^i(a),
\end{equation}
where $\sigma(a) = a^q$ for $a$ in $\fnum_{q^n}$, is the Frobinus of $\fnum_{q^n}$. If $A$ is a polynomial in $\fnum_{q^n}[x]$, $A=a_kx^k + a_{k-1}x^{k-1} + \cdots + a_1x + a_0$, the trace map and norm can be extended to $\fnum_{q^n}[x]$ by
\begin{equation}
  \label{eq:trNex}
    \tr(A) = \sum_{i=1}^n \sigma^i(A), \mbox{ and } \norm(A) = \prod_{i=1}^n \sigma^i(A),
\end{equation}
where $\sigma(A) = \sum_{i=0}^k \sigma(a_i)x^i$. 

For a polynomial $H$ in $\fnum_q[x]$ and, therefore, also a polynomial in $\fnum_{q^n}[x]$. To define a Gauss sum modulo $H$ on $\fnum_{q^n}[x]$, for any $A$ in $\fnum_{q^n}[x]$, we set
\begin{equation}
  \label{eq:psiG}
  \psi_G^{(n)}(A) = \psi_G(\tr(A)), \mbox{ and }\chi^{(n)}(A) = \chi(N(A)),
\end{equation}
thus the Gauss sums $G(\chi^{(n)}, \psi_G^{(n)})$ modulo $H$ on $\fnum_{q^n}[x]$ is given by
\begin{equation}
  \label{eq:GaussEx}
  \psi(\chi^{(n)}, \psi_{G}^{(n)}) = \sum_{\substack{D\in \fnum_{q^n}[x]\\ D \moda H}}\chi^{(n)}(D)\psi_G^{(n)}(D),
\end{equation}
where the summation extends over a complete residue system modulo $H$ in $\fnum_{q^n}[x]$.

By the above notations, we may define a polynomial Ramanujan sum $\eta^{(n)}(G, H)$ modulo $H$ on $\fnum_{q^n}[x]$ by
\begin{equation}
  \label{eq:RamaEx}
  \eta^{(n)}(G, H) = \sum_{\substack{D\moda H\\ D\in \fnum_{q^n}[x]}}\chi_0^{(n)}(D)\psi_G^{(n)}(D),
\end{equation}
and a generalized version $\eta^{(n)}_k(G, H)$ by

\begin{equation}
  \label{eq:RamaExGen}
  \eta_k^{(n)}(G, H) = \sum_{\substack{D\moda H^k\\(D, H^k)_k=1}}\psi_G^{(n)}(D), 138
\end{equation}
where the summation ranges over a complete residue system modulo $H^k$ in $\fnum_{q^n}[x]$. 

If $\chi$ and $\psi_G$ not both are principal, in \cite{ref:49} we showed the following Davenport--Hasse type formula
\begin{equation}
  \label{eq:Daven-Hass}
  (-1)^{m-m_1}\frac{\phi^{(n)}(N)}{\phi^{(n)}(H)}G(\chi^{(n)}, \psi_G^{(n)})  = \left((-1)^{m-m_1}\frac{\phi(N)}{\phi(H)}G(\chi, \psi_G) \right)^n,
\end{equation}
where $\phi(H)$ is the Euler totient function on $\fnum_{q}[x]$, $\phi^{(n)}(H)$ is the function on $\fnum_{q^n}[x]$, $N=\frac{H}{(G, H)}$, $m=\deg(H)$, and $m_1=\deg(G, H)$. 

As a direct consequence of \eqref{eq:Daven-Hass}, if $H \nmid G$, then $\psi_G$ is not principal, and we have
\begin{equation}
  \label{eq:Daven-HassCon}
  (-1)^{m-m_1}\frac{\phi^{(n)}(N)}{\phi^{(n)}(H)}\eta^{(n)}(G, H)  = \left((-1)^{m-m_1}\frac{\phi(N)}{\phi(H)}\eta(G, H) \right)^n.
\end{equation}

The main purpose of this section is to show that the generalized version $\eta_k(H, G)$ also shares this kind of Davenport--Hasse type formula. We have
\begin{theorem}
  \label{thm:71}
If $H$ and $G$ are any polynomials in $\fnum_q[x]$ such that $H^k \nmid G$ and $H \neq 0$, then
\begin{equation}
  \label{eq:thm71}
  (-1)^{m-m_1}\frac{\phi_k^{(n)}(N)}{\phi_k^{(n)}(H)}\eta_k^{(n)}(G, H)  = \left((-1)^{m-m_1}\frac{\phi_k(N)}{\phi_k(H)}\eta_k(G, H) \right)^n,
\end{equation}
where $\phi_k(H)$ is the Jordan totient function on $\fnum_q[x]$, and $\phi_k^{(n)}(H)$ is the function on $\fnum_{q^n}[x]$, $N=\frac{H}{A}$, $A^k = (G, H^k)_k$, $m=\deg(H)$, and $m_1 = \deg(A)$. 
\end{theorem}

\begin{proof}
  If $A^k = (G, H^k)_k$ in $\fnum_q[x]$, it is easy to verify that $A^k = (G, H^k)_k$ holds in $\fnum_{q^n}[x]$. By \eqref{eq:thmHolder}, we have
  \begin{equation}
    \label{eq:thm71p1}
    \eta_k^{(n)}(G, H) = \phi_k^{(n)}(H)\mu^{(n)}(N)\left(\phi_k^{(n)}(N) \right)^{-1},
  \end{equation}
and 
  \begin{equation}
    \label{eq:thm71p2}
    \eta_k(G, H) = \phi_k(H)\mu(N)\phi_k^{-1}(N),
  \end{equation}
where $\mu^{(n)}(H)$ is the M\"{o}bius function on $\fnum_{q^n}[x]$. To prove \eqref{eq:thm71}, it suffices to show that
\begin{equation}
  \label{eq:thm71p2}
  (-1)^{m+m_1}\mu^{(n)}(N) = \left((-1)^{m+m_1}\mu(N) \right)^n,
\end{equation}
where $N=\frac{H}{A}$, and $A^k = (G, H^k)_k$. 

We note that both sides of \eqref{eq:thm71p2} are multiplicative in $H$, so it suffices to prove \eqref{eq:thm71p2} when $H=P^t$, where $P$ is an irreducible in $\fnum_q[x]$. Let $\deg(P)=h$. Since $H^k\nmid G$, then $t\geq 1$. If $A=P^{t_1}$ with $t_1< t-1$, then both sides of \eqref{eq:thm71p2} are zero. Therefore, we may suppose $A=P^{t-1}$, and $N=P$. It is well-known that $P$ is product of exactly $(h, n)$ irreducibles in $\fnum_{q^n}[x]$, so \eqref{eq:thm71p2} becomes that
\begin{equation}
  \label{eq:thm71p3}
  (-1)^{th+(t-1)h+(h, n)} =  (-1)^{n(th+(t-1)h+1)},
\end{equation}
which is equivalent to
\begin{equation}
  \label{eq:thm71p4}
  h + (h, n) \equiv n(h+1) ~(\moda 2).
\end{equation}
It is easy to verify that \eqref{eq:thm71p4} is true for any positive integers $n$ and $h$, and we complete the proof of Theorem~\ref{thm:71}.
\end{proof}


\end{document}